\documentclass[reqno, 12pt]{amsart}
\usepackage{amsmath, amsthm, amssymb, mathtools, graphicx, paralist, cite, color}
\usepackage[all]{xypic}
\definecolor{e-mail}{rgb}{0,.40,.80}
\definecolor{reference}{rgb}{.20,.60,.22}
\definecolor{citation}{rgb}{0,.40,.80}
\usepackage[pdfstartview=FitH, colorlinks, linkcolor=reference, citecolor=citation, urlcolor=e-mail]{hyperref}

\newtheorem{thm}{Theorem}
\newtheorem*{mthm}{\protect{Theorem~\ref{thm:main}}}
\newtheorem{cor}[thm]{Corollary}
\newtheorem{lem}[thm]{Lemma}
\newtheorem{prop}[thm]{Proposition}

\theoremstyle{definition}
\newtheorem{defn}[thm]{Definition}
\theoremstyle{remark}
\newtheorem{rem}[thm]{Remark}
\numberwithin{thm}{section}
\theoremstyle{definition}

\theoremstyle{definition}

\newtheorem{eg}[thm]{Example}
\theoremstyle{definition}
\numberwithin{equation}{section}

\newcommand{\K}{\mathbb K} 
\newcommand{\ie}{{\it  i.e.}}

\title{Twisted Mahler discrete residues}

\author{Carlos E. Arreche}
\address{Department of Mathematical Sciences, The University of Texas at Dallas, Richardson, TX 75080}
\email{arreche@utdallas.edu}
 \thanks{The work of C.E.~Arreche was partially supported by NSF grant CCF-1815108.} 

\author{Yi Zhang}
\address{Department of Foundational  Mathematics, School of Mathematics and Physics, Xi'an Jiaotong-Liverpool University, 
 Suzhou, 215123, China}
\email{Yi.Zhang03@xjtlu.edu.cn}
\thanks{The work of Y.~Zhang was supported by the NSFC Young Scientist Fund No.\ 12101506, the Natural Science Foundation of the Jiangsu Higher Education Institutions of China No.\ 21KJB110032, and XJTLU Research Development Fund No.\ RDF-20-01-12.}

\date{\today}

\begin{document}

\begin{abstract}
Recently we constructed Mahler discrete residues for rational functions and showed they comprise a complete obstruction to the Mahler summability problem of deciding whether a given rational function $f(x)$ is of the form $g(x^p)-g(x)$ for some rational function $g(x)$ and an integer $p > 1$. Here we develop a notion of $\lambda$-twisted Mahler discrete residues for $\lambda\in\mathbb{Z}$, and show that they similarly comprise a complete obstruction to the twisted Mahler summability problem of deciding whether a given rational function $f(x)$ is of the form $p^\lambda g(x^p)-g(x)$ for some rational function $g(x)$ and an integer $p>1$. We provide some initial applications of twisted Mahler discrete residues to differential creative telescoping problems for Mahler functions and to the differential Galois theory of linear Mahler equations. 
\end{abstract}


\maketitle

\tableofcontents


\section{Introduction} \label{sec:introduction}
 Continuous residues are fundamental and crucial tools in complex analysis, and have extensive and compelling applications in combinatorics~\cite{Flajolet_Sedgewick}. In the last decade, a theory of discrete and $q$-discrete residues was proposed in~\cite{ChenSinger2012} for the study of telescoping problems for bivariate rational functions, and subsequently found applications in the computation of differential Galois groups of second-order linear difference~\cite{arreche:2017} and $q$-difference equations~\cite{ArrecheZhang2020} and other closely-related problems \cite{Chen:2018,HouWang:2015}. More recently, the authors of~\cite{Caruso2021, Caruso-Durand:2021} developed a theory of residues for skew rational functions, which has important applications in duals of linearized Reed-Solomon codes~\cite{Caruso-Durand:2021}. In \cite{HardouinSinger2021} the authors introduce a notion of elliptic orbit residues 
 which, in analogy with \cite{ChenSinger2012}, similarly serves as a complete obstruction to summability in the context of elliptic shift difference operators. In~\cite{arreche-zhang:2022} we initiated a theory of Mahler discrete residues aimed at helping bring to the Mahler case the successes of these earlier notions of residues.

 Let $\K$ be an algebraically closed field of characteristic zero and $\K(x)$ be the field of rational functions in an indeterminate $x$ over $\K$. Fix an integer $p\geq 2$. For a given $f(x)\in\K(x)$, we considered in \cite{arreche-zhang:2022} the \emph{Mahler summability problem} of deciding effectively whether $f(x)=g(x^p)-g(x)$ for some $g(x)\in\mathbb{K}(x)$; if so, we say $f(x)$ is \emph{Mahler summable}. We defined in \cite{arreche-zhang:2022} a collection of $\mathbb{K}$-vectors, called \emph{Mahler discrete residues} of $f(x)$ and defined purely in terms of its partial fraction decomposition, having the property that they are all zero if and only if $f(x)$ is Mahler summable.

 More generally, a (linear) \emph{Mahler equation} is any equation of the form \begin{equation}\label{eq:mahler-intro}y(x^{p^n})+a_{n-1}(x)y(x^{p^{n-1}})+\dots+a_1(x)y(x^p)++a_0(x)y(x)=0,\end{equation} where the $a_i(x)\in\mathbb{K}(x)$ and $y(x)$ is an unknown ``function'' (or possibly some more general entity, e.g., the generating series of a combinatorial object, a Puisseux series, etc.). The motivation to study Mahler equations in general comes from several directions. They first arose in \cite{Mahler1929} in connection with transcendence results on values of special functions at algebraic numbers, and have since found other applications to automata theory and automatic sequences since the work of \cite{cobham:1968}. We refer to~\cite{AB:2017,dhr-m, Dreyfus2018, adh21} and the references therein for more details. We also mention that a different (and, for some purposes, better) approach to the Mahler summability problem is contained in \cite{Dreyfus2018}, where the authors develop efficient algorithms to find, in particular, all the rational solutions to a linear Mahler equation. Thus \cite{Dreyfus2018} decides efficiently whether any \emph{given} $f(x)\in\mathbb{K}(x)$ is Mahler summable: namely, by either actually finding the corresponding certificate $g(x)\in \K(x)$ such that $f(x)=g(x^p)-g(x)$ if it exists or else deciding that there is no such $g(x)\in\K(x)$. We emphasize that, in contrast, the approach undertaken in \cite{arreche-zhang:2022} is obstruction-theoretic, with the upshot that it spells out (theoretically) exactly what it takes for any $f(x)\in\K(x)$ whatsoever to be Mahler summable or not, but with the drawback that it is likely to be infeasible in practice for all but the simplest/smallest choices of $f(x)$. All the same, the approach initiated in \cite{arreche-zhang:2022}, and continued in the present work, is a worthwhile and useful complement to that of \cite{Dreyfus2018} --- not only because of the theoretical questions that it answers for the first time, but moreover also because of its practical implications.

A particularly fruitful approach over the last few decades to study difference equations in general, and Mahler equations such as \eqref{eq:mahler-intro} in particular, is through the Galois theory for linear difference equations developed in \cite{vanderput-singer:1997}, and the differential (also sometimes called parameterized) Galois theory for difference equations developed in \cite{HardouinSinger2008}. Both theories associate a geometric object to a given difference equation such as \eqref{eq:mahler-intro}, called the \emph{Galois group}, that encodes the sought (differential-)algebraic properties of the solutions to the equation. There are now several algorithms and theoretical results (see in particular \cite{roques:2017,dhr-m,arreche-singer:2016,arreche-dreyfus-roques:2019}) addressing  qualitative questions about solutions of Mahler equations \eqref{eq:mahler-intro}, in particular whether they must be (differentially) transcendental, which rely on procedures to compute ``enough'' information about the corresponding Galois group (i.e., whether it is ``sufficiently large''). These Galois-theoretic arguments very often involve, as a sub-problem, deciding whether a certain auxiliary object (often but not always a rational solution to some Riccati-type equation) is Mahler summable (possibly after applying some linear differential operator to it, i.e., a telescoper). Rather than being able to answer the Mahler summability question for any one individual rational function, the systematic obstructions to the Mahler summability problems developed here serve as essential building blocks for other results and algorithms that rely on determining Mahler summability as an intermediate step. An immediate application of the technology of the technology developed here is Proposition~\ref{prop:nishioka}, which has the following concrete consequence (when paired with the results of \cite[Theorem~1.3]{adh21}): if $y_1(x),\dots,y_t(x)\in\mathbb{K}((x))$ are Laurent series solutions to Mahler equations of the form \[y_i(x^p)=a_i(x)y_i(x)\] for some non-zero $a_i(x)\in\mathbb{K}(x)$, then either the $y_1(x),\dots,y_t(x)$ are differentially independent over $\mathbb{K}(x)$ or else they are multiplicatively dependent over $\mathbb{K}(x)^\times$, i.e., there exist integers $k_1,\dots,k_t\in\mathbb{Z}$, not all zero, such that $\prod_{i=1}^ty_i(x)^{k_i}\in\mathbb{K}(x)$. Let us explain in more detail the technology that we develop.

For arbitrary $\lambda\in\mathbb{Z}$ and $f(x)\in\mathbb{K}(x)$, we say that $f(x)$ is $\lambda$\emph{-Mahler summable} if there exists $g(x)\in\mathbb{K}(x)$ such that $f(x)=p^\lambda g(x^p)-g(x)$. We shall construct certain $\mathbb{K}$-vectors from the partial fraction decomposition of $f(x)$, which we call the \emph{(twisted)} $\lambda$-\emph{Mahler discrete residues} of $f(x)$, and prove our main result in Section~\ref{sec:proof}:
\begin{thm}\label{thm:main} For $\lambda\in\mathbb{Z}$, $f(x)\in\mathbb{K}(x)$ is $\lambda$-Mahler summable if and only if every $\lambda$-Mahler discrete residue of $f$ is zero.
\end{thm} 

Our desire to develop an obstruction theory for such a ``twisted'' $\lambda$-Mahler summability problem, beyond the ``un-twisted'' $0$-Mahler summability problem considered in \cite{arreche-zhang:2022}, is motivated by our desire to apply this obstruction theory to the following kind of \emph{Mahler creative telescoping problem}. Given $f_1,\dots,f_n\in\K(x)$ decide whether there exist linear differential operators $\mathcal{L}_1,\dots,\mathcal{L}_n\in\K[\delta]$, for $\delta$ some suitable derivation, such that $\mathcal{L}_1(f_1)+\dots+\mathcal{L}_n(f_n)$ is suitably Mahler summable. The unfortunately vague (but deliberate) double-usage of ``suitable'' above is due to the fact that there are in the Mahler case two traditional and respectable ways to adjoin a ``Mahler-compatible'' derivation in order to study differential-algebraic properties of solutions of Mahler equations, as we next explain and recall.

A $\sigma\delta$-field is a field equipped with an endomorphism $\sigma$ and a derivation $\delta$ such that $\sigma\circ\delta=\delta\circ\sigma$. Such are the basefields considered in the $\delta$-Galois theory for linear $\sigma$-equations developed in \cite{HardouinSinger2008}. Denoting by $\sigma:\K(x)\rightarrow\K(x):f(x)\mapsto f(x^p)$ the \emph{Mahler endomorphism}, one can show there is no non-trivial derivation $\delta$ on $\K(x)$ that commutes with this $\sigma$. In the literature one finds the following two approaches (often used in combination; see e.g. \cite{dhr-m,adh21}): (1) take $\delta=x\frac{d}{dx}$, and find a systematic way to deal with the fact that $\sigma$ and $\delta$ do not quite commute (but almost do), $\sigma\circ\delta=p\, \delta\circ\sigma$; or (2) work over the larger field $\K(x, \log x)$, where $\sigma(\log x)=p\log x$, and set $\delta=x\log x\frac{d}{dx}$, and find a systematic way to deal with this new element $\log x$ as the cost of having $\sigma\circ\delta=\delta\circ\sigma$ on the nose. There is, to be sure, a dictionary of sorts between these two approaches. We postpone a more careful discussion of these issues until it becomes absolutely necessary in Section~\ref{sec:telescoping}, except to adopt the latter approach in this introduction to briefly motivate the centrality of the $\lambda$-Mahler summability problems for arbitrary $\lambda\in\mathbb{Z}$ in the differential study of Mahler functions.

Let us consider the $\sigma\delta$-field $L:=\K(x, \log x)$, and given $F\in L$, let us write the $\log$-Laurent series expansion \[F=\sum_{\lambda\geq N}f_\lambda(x)\log^\lambda x\in \K(x)((\log x)),\] where $f_\lambda(x)\in\K(x)$ for each $\lambda\in \mathbb{Z}$, and $\log^\lambda x:=[\log x]^\lambda$. Let us suppose that there exists $G\in \hat{L}:=\K(x)((\log x))$ such that $F=\sigma(G)-G$ (where $\sigma$ is applied term-by-term). Writing such a putative $G=\sum_{\lambda\geq N} g_\lambda(x)\log^\lambda x \in \hat{L}$, for some $g_\lambda(x)\in\K(x)$ for $\lambda\in\mathbb{Z}$, we find that $F$ is Mahler summable within $\hat{L}$ if and only if each $f_\lambda(x)=p^\lambda g_\lambda(x^p)-g(x)$ for each $\lambda\in \mathbb{Z}$. 
 
Our strategy expands upon that of \cite{arreche-zhang:2022}, which in turn was inspired by that of \cite{ChenSinger2012}: for $\lambda\in\mathbb{Z}$, we utilize the coefficients occurring in the partial fraction decomposition of $f(x)$ to construct in Section~\ref{sec:reduction} a $\lambda$\emph{-Mahler reduction} $\bar{f}_\lambda(x)\in\mathbb{K}(x)$ such that \begin{equation}\label{eq:mahler-reduction}\bar{f}_\lambda(x)=f(x)+\bigl(p^\lambda g_\lambda(x^p)-g_\lambda(x)\bigr)\end{equation} for some $g_\lambda(x)\in\mathbb{K}(x)$ (whose explicit computation it is our purpose to avoid!), with the structure of this $\bar{f}_\lambda(x)$ being such that it cannot possibly be $\lambda$-Mahler summable unless $\bar{f}_\lambda(x)=0$. The $\lambda$-Mahler discrete residues of $f(x)$ are (vectors whose components are) the coefficients occurring in the partial fraction decomposition of $\bar{f}_\lambda(x)$. This $\bar{f}_\lambda(x)$ plays the role of a ``$\lambda$-Mahler remainder'' of $f(x)$, analogous to the remainder of Hermite reduction in the context of integration. 


\section{Preliminaries} \label{sec:preliminaries}
In this section we recall and expand upon some conventions, notions, and ancillary results from \cite{arreche-zhang:2022} that we shall use systematically throughout this work.

\subsection{Notation and conventions}

We fix once and for all an algebraically closed field $\mathbb{K}$ of characteristic zero and an integer $p\geq 2$ (not necessarily prime). We denote by $\mathbb{K}(x)$ the field of rational functions in the indeterminate $x$ with coefficients in $\mathbb{K}$. We denote by $\sigma:\mathbb{K}(x)\rightarrow \mathbb{K}(x)$ the $\mathbb{K}$-linear endomorphism defined by $\sigma(x)=x^p$, called the \emph{Mahler operator}, so that $\sigma(f(x))=f(x^p)$ for $f(x)\in\mathbb{K}(x)$. For $\lambda\in\mathbb{Z}$, we write $\Delta_\lambda:=p^\lambda\sigma-\mathrm{id}$, so that $\Delta_\lambda(f(x))=p^\lambda f(x^p)-f(x)$ for $f(x)\in\mathbb{K}(x)$. We often suppress the functional notation and write simply $f\in\mathbb{K}(x)$ instead of $f(x)$ whenever no confusion is likely to arise. We say that $f\in \mathbb{K}(x)$ is $\lambda$\emph{-Mahler summable} if there exists $g\in\mathbb{K}(x)$ such that $f=\Delta_\lambda(g)$.

Let $\mathbb{K}^\times=\mathbb{K}\backslash\{0\}$ denote the multiplicative group of $\mathbb{K}$. Let $\mathbb{K}^\times_t$ denote the torsion subgroup of $\mathbb{K}^\times$, \ie, the group of roots of unity in $\mathbb{K}^\times$. For $\zeta\in\mathbb{K}^\times_t$, the \emph{order} of $\zeta$ is the smallest $r\in\mathbb{N}$ such that $\zeta^r=1$. We fix once and for all a compatible system of $p$-power roots of unity $(\zeta_{p^n})_{n\geq 0} \subset \mathbb{K}^\times_t$, that is, each $\zeta_{p^n}$ has order $p^n$ and $\zeta_{p^n}^{p^\ell}=\zeta_{p^{n-\ell}}$ for $0\leq\ell\leq n$. 
Each $f\in\mathbb{K}(x)$ decomposes uniquely as \begin{equation}\label{eq:f-decomposition}f=f_{\infty}+f_\mathcal{T},\end{equation} where $f_{\infty}\in\mathbb{K}[x,x^{-1}]$ is a Laurent polynomial and $f_\mathcal{T}=\frac{a}{b}$ for polynomials $a,b\in\mathbb{K}[x]$ such that either $a=0$ or else $\mathrm{deg}(a)<\mathrm{deg}(b)$ and $\mathrm{gcd}(a,b)=1=\mathrm{gcd}(x,b)$. The reasoning behind our choice of subscripts $\infty$ and $\mathcal{T}$ for the Laurent polynomial component of $f$ and its complement will become apparent in the sequel. 

\begin{lem}\label{lem:rational-decomposition} The $\mathbb{K}$-linear decomposition $\mathbb{K}(x)\simeq \mathbb{K}[x,x^{-1}]\oplus \mathbb{K}(x)_\mathcal{T}$ given by $f\leftrightarrow f_{\infty}\oplus f_\mathcal{T}$ as in \eqref{eq:f-decomposition} is $\sigma$-stable. For $f,g\in\mathbb{K}(x)$ and for $\lambda\in\mathbb{Z}$, $f=\Delta_\lambda(g)$ if and only if $f_{\infty}=\Delta_\lambda(g_{\infty})$ and $f_\mathcal{T}=\Delta_\lambda(g_\mathcal{T})$.
\end{lem}


\subsection{Mahler trajectories, Mahler trees, and Mahler~cycles}\label{sec:trajectories-trees-cycles}

We let \[\mathcal{P}:=\{p^n \ | \ n\in\mathbb{Z}_{\geq 0}\}\] denote the multiplicative monoid of non-negative powers of $p$. Then $\mathcal{P}$ acts on $\mathbb{Z}$ by multiplication, and the set of \emph{maximal trajectories} for this action is \[\mathbb{Z}/\mathcal{P} := \bigl\{\{0\}\bigr\} \cup \bigl\{\{ip^n \ | \ n\in\mathbb{Z}_{\geq 0}\} \ \big| \ i\in\mathbb{Z} \ \text{such that} \ p\nmid i \bigr\}.\]

\begin{defn}\label{defn:trajectory-projection}
For a maximal trajectory $\theta\in\mathbb{Z}/\mathcal{P}$, we let \begin{equation}\label{eq:trajectory-subspace} \textstyle\mathbb{K}[x,x^{-1}]_\theta := \left\{\sum_j c_j x^j\in\mathbb{K}[x,x^{-1}] \ \middle| \ c_j=0 \ \text{for all} \ j \notin \theta\right\},\end{equation} and call it the $\theta$\emph{-subspace}. The $\theta$\emph{-component} $f_\theta$ of $f\in\mathbb{K}(x)$ is the projection of the component $f_{\infty}$ of $f$ in \eqref{eq:f-decomposition} to $\mathbb{K}[x,x^{-1}]_\theta$ as in~\eqref{eq:trajectory-subspace}.\end{defn}

We obtain similarly as in \cite[Lem.~2.3]{arreche-zhang:2022} 
the following result.

\begin{lem}\label{lem:trajectory-projection} For $f,g\in\mathbb{K}(x)$ and for $\lambda\in\mathbb{Z}$, $f_{\infty}=\Delta_\lambda(g_{\infty})$ if and only if $f_\theta=\Delta_\lambda(g_\theta)$ for every maximal trajectory $\theta\in\mathbb{Z}/\mathcal{P}$.\end{lem}

\begin{defn} \label{defn:trees} We denote by $\mathcal{T}$ the set of equivalence classes in $\mathbb{K}^\times$ for the equivalence relation $\alpha\sim \gamma\Leftrightarrow\alpha^{p^r}=\gamma^{p^s}$ for some $r,s\in\mathbb{Z}_{\geq 0}$. For $\alpha\in\mathbb{K}^\times$, we denote by $\tau(\alpha)\in\mathcal{T}$ the equivalence class of $\alpha$ under $\sim$. The elements $\tau\in\mathcal{T}$ are called \emph{Mahler trees}.
\end{defn}

We refer to \cite[Remark~2.7]{arreche-zhang:2022} for a brief discussion on our choice of nomenclature in Definition~\ref{defn:trees}.

\begin{defn}\label{defn:tree-projection}
For a Mahler tree $\tau\in\mathcal{T}$, the $\tau$\emph{-subspace} is \begin{equation}\label{eq:tree-subspace}\mathbb{K}(x)_\tau := \bigl\{f_\mathcal{T}\in \mathbb{K}(x)_\mathcal{T} \ \big| \ \text{every pole of} \ f_\mathcal{T} \ \text{is contained in} \ \tau \}.\end{equation} For $f\in\mathbb{K}(x)$, the $\tau$\emph{-component} $f_\tau$ of $f$ is the projection of the component $f_\mathcal{T}$ of $f$ in \eqref{eq:f-decomposition} to the $\tau$-subspace $\mathbb{K}(x)_\tau$ in \eqref{eq:tree-subspace}.\end{defn}

The following result is proved similarly as in \cite[Lem.~2.12]{arreche-zhang:2022}.

\begin{lem}\label{lem:tree-decomposition} For $f,g\in\mathbb{K}(x)$ and for $\lambda\in\mathbb{Z}$, $f_\mathcal{T}=\Delta_\lambda(g_\mathcal{T})$ if and only if $f_\tau=\Delta_\lambda(g_\tau\!)$ for every Mahler tree $\tau\in\mathcal{T}$.
\end{lem}

\begin{defn} \label{defn:cycles} For a Mahler tree $\tau\in\mathcal{T}$, the (possibly empty) \emph{Mahler cycle} of $\tau$ is \[\mathcal{C}(\tau):=\{\gamma \in \tau \ | \ \gamma \ \text{is a root of unity of order coprime to } p\}.\] The (possibly zero) \emph{cycle length} of $\tau$ is defined to be $\varepsilon(\tau):=|\mathcal{C}(\tau)|$.

For $e\in\mathbb{Z}_{\geq 0}$, we write $\mathcal{T}_e:=\{\tau\in \mathcal{T} \ | \ \varepsilon(\tau)=e\}$. We refer to $\mathcal{T}_0$ as the set of \emph{non-torsion Mahler trees}, and to $\mathcal{T}_+:=\mathcal{T}-\mathcal{T}_0$ as the set of \emph{torsion Mahler trees}.
\end{defn}

\begin{rem}\label{rem:cycle-facts} Let us collect as in \cite[Rem.~2.10]{arreche-zhang:2022} some immediate observations about Mahler cycles that we shall use, and refer to, throughout the sequel.

For $\tau\in\mathcal{T}$ it follows from the Definition~\ref{defn:trees} that either $\tau\subset\mathbb{K}^\times_t$ or else $\tau\cap\mathbb{K}^\times_t=\emptyset$ (that is, either $\tau$ consists entirely of roots of unity or else $\tau$ contains no roots of unity at all). In particular, $\tau\cap\mathbb{K}^\times_t=\emptyset\Rightarrow\mathcal{C}(\tau)=\emptyset\Leftrightarrow \varepsilon(\tau)=0\Leftrightarrow \tau\in\mathcal{T}_0$ (the \emph{non-torsion case}).

On the other hand, $\mathbb{K}^\times_t$ consists of the pre-periodic points for the action of the monoid $\mathcal{P}$ on $\mathbb{K}^\times$ given by $\alpha\mapsto \alpha^{p^n}$ for $n\in\mathbb{Z}_{\geq 0}$. For $\tau\subset\mathbb{K}^\times_t$ (the \emph{torsion case}), the Mahler cycle $\mathcal{C}(\tau)$ is a non-empty set endowed with a simply transitive action of the quotient monoid $\mathcal{P}/\mathcal{P}^e\simeq\mathbb{Z}/e\mathbb{Z}$, where $\mathcal{P}^e:=\{p^{ne} \ | \ n\in\mathbb{Z}\}$, and $e:=\varepsilon(\tau)$. We emphasize that in general $\mathcal{C}(\tau)$ is only a set, and not a group. The Mahler tree $\tau(1)$ consists precisely of the roots of unity $\zeta\in\mathbb{K}^\times_t$ whose order $r$ is such that $\mathrm{gcd}(r,p^n)=r$ for some $p^n\in\mathcal{P}$, or equivalently such that every prime factor of $r$ divides $p$. When $\tau\subset\mathbb{K}^\times_t$ but $\tau\neq \tau(1)$, the cycle length $\varepsilon(\tau)=e$ is the order of $p$ in the group of units $(\mathbb{Z}/r\mathbb{Z})^\times$, where $r>1$ is the common order of the roots of unity $\gamma\in\mathcal{C}(\tau)$, and $\mathcal{C}(\tau)=\{\gamma^{p^\ell} \ | \ 0\leq \ell\leq e-1\}$ for any given $\gamma\in\mathcal{C}(\tau)$. We shall often abusively write $\mathcal{C}(\tau)=\{\gamma^{p^\ell} \ | \ \ell\in\mathbb{Z}/e\mathbb{Z}\}$.
\end{rem}


\subsection{Mahler supports and singular supports in Mahler trees}\label{sec:forest-support}

Mahler trees allow us to define the following bespoke variants of the singular support $\mathrm{sing}(f)$ of a rational function $f$ (\ie, its set of poles) and the order $\mathrm{ord}_\alpha(f)$ of a pole of $f$ at $\alpha\in\mathbb{K}$,
which are particularly well-suited to the Mahler context.

\begin{defn} \label{defn:supp} For $f\in \mathbb{K}(x)$, we define $\mathrm{supp}(f)\subset \mathcal{T}\cup\{\infty\}$, called the \emph{Mahler support} of $f$, as follows:
\begin{itemize}
\item $\infty\in\mathrm{supp}(f)$ if and only if $f_{\infty}\neq 0$; and
\item for $\tau\in\mathcal{T}$, $\tau\in\mathrm{supp}(f)$ if and only if $\tau$ contains a pole of $f$.
\end{itemize}

For $\tau\in\mathcal{T}$, the \emph{singular support} of $f$ in $\tau$, denoted by $\mathrm{sing}(f,\tau)$, is the (possibly empty) set of poles of $f$ contained in $\tau$, and the \emph{order} of $f$ at $\tau$ is \[\mathrm{ord}(f,\tau):=\mathrm{max}\bigl(\{0\}\cup\left\{\mathrm{ord}_\alpha(f) \ \middle| \ \alpha\in\mathrm{sing}(f,\tau)\right\}\bigr).\]
\end{defn}

For the sake of completeness, we include the straightforward proof of the following lemma, which was omitted from \cite[Section~2.2]{arreche-zhang:2022} for lack of space.

\begin{lem}\label{lem:stable-support} For $f,g\in\mathbb{K}(x)$, $\tau\in\mathcal{T}$, $\lambda\in\mathbb{Z}$, and $0\neq c\in\mathbb{K}$, we have the following:
\begin{enumerate}
\item $\mathrm{supp}(f)=\emptyset\Longleftrightarrow f=0$;
\item $\mathrm{supp}(\sigma(f))=\mathrm{supp}(f)=\mathrm{supp}(c\cdot f)$; and
\item $\mathrm{supp}(f+g)\subseteq \mathrm{supp}(f)\cup\mathrm{supp}(g)$.
\item $\tau\in\mathrm{supp}(\Delta_\lambda(g)) \Longleftrightarrow \tau\in\mathrm{supp}(g)$;
\item $\mathrm{ord}(\sigma(f),\tau)=\mathrm{ord}(f,\tau)=\mathrm{ord}(c\cdot f,\tau)$;
\item $\mathrm{ord}(f+g,\tau)\leq\mathrm{max}\bigl(\mathrm{ord}(f,\tau),\mathrm{ord}(g,\tau)\bigr)$; and
\item $\mathrm{ord}(\Delta_\lambda(g),\tau)=\mathrm{ord}(g,\tau)$.
\end{enumerate}
\end{lem}

\begin{proof}
(1).~$f=0\Longleftrightarrow f_{\infty}=0$ and $f_\mathcal{T}=0$, and $f_\mathcal{T}=0\Longleftrightarrow f$ has no poles in $\mathbb{K}^\times$.

(2) and (5).~For $0\neq c\in \mathbb{K}$, $cf_{\infty}\neq 0$ if and only if $f_{\infty}\neq 0$, and $f$ and $cf$ have the same poles and the orders of these poles are the same, and therefore $\mathrm{supp}(f)=\mathrm{supp}(cf)$ and $\mathrm{ord}(f,\tau)=\mathrm{ord}(cf,\tau)$ for every $\tau\in\mathcal{T}$. Moreover, $\sigma(f_{\infty})\neq 0$ if and only if $f_{\infty}\neq 0$, since $\sigma$ is an injective endomorphism of $\K(x)$, and $\alpha\in\mathbb{K}^\times$ is a pole of $\sigma(f)$ if and only if $\alpha^p$ is a pole of $f$, whence $\tau$ contains a pole of $f$ if and only if $\tau$ contains a pole of $\sigma(f)$. In this case, it is clear that $\mathrm{ord}(\sigma(f),\tau)\leq\mathrm{ord}(f,\tau)$. Moreover, since $f$ has only finitely many poles in $\tau$ of maximal order $m:=\mathrm{ord}(f,\tau)$, there exists $\alpha\in\mathrm{sing}(\sigma(f),\tau)$ such that $\mathrm{ord}_{\alpha^p}(f)=m>\mathrm{ord}_\alpha(f)$, and it follows that $\mathrm{ord}_\alpha(\sigma(f))=m=\mathrm{ord}(\sigma(f),\tau)$.

(3) and (6).~If $f_{\infty}+g_{\infty}\neq 0$ then at least one of $f_{\infty}\neq 0$ or $g_{\infty}\neq 0$. The set of poles of $f+g$ is contained in the union of the set of poles of $f$ and the set of poles of $g$, and therefore if $\tau$ contains a pole of $f+g$ then $\tau$ must contain a pole of $f$ or a pole of $g$. This shows that $\mathrm{supp}(f+g)\subseteq\mathrm{supp}(f)\cup\mathrm{supp}(g)$. For $m$ the maximal order of a pole of $f+g$ in $\tau$ we see that at least one of $f$ or $g$ must contain a pole of order $m$ in $\tau$. This shows that $\mathrm{ord}(f+g,\tau)\leq\max(\mathrm{ord}(f,\tau),\mathrm{ord}(g,\tau))$.

(4) and (7).~By~(2) and (3), $\mathrm{supp}(\Delta_\lambda(g))\subseteq\mathrm{supp}(g)$, and by~(5) and~(6), $\mathrm{ord}(\Delta_\lambda(g),\tau)\leq\mathrm{ord}(g,\tau)$. Suppose $\tau\in\mathrm{supp}(g)$, and let $\alpha_1,\dots,\alpha_s\in\mathrm{sing}(g,\tau)$ be all the elements, pairwise distinct, with $\mathrm{ord}_{\alpha_j}(g)=\mathrm{ord}(g,\tau)=:m\geq 1$, and choose $\gamma_j\in \tau$ such that $\gamma_j^p=\alpha_j$, we find as in the proof of (5) that $\mathrm{ord}_{\zeta_p^i\gamma_j}(\sigma(g))=m$ and the elements $\zeta_p^i\gamma_j$ are pairwise distinct for $0\leq i\leq p-1$ and $1\leq j\leq s$, whence at least one of the $\zeta_p^i\gamma_j$ is different from every $\alpha_{j'}$ for $1\leq j'\leq s$, and therefore $\mathrm{ord}(\Delta_\lambda(g),\tau)=m$, which implies in particular that $\tau\in\mathrm{supp}(\Delta_\lambda(g))$.
\end{proof}


\subsection{Mahler dispersion}\label{sec:dispersion}

We now recall from \cite{arreche-zhang:2022} the following Mahler variant of the notion of (polar) dispersion used in \cite{ChenSinger2012}, following the original definitions in \cite{Abramov:1971,Abramov:1974}.

\begin{defn} \label{defn:dispersion} For $f\in \mathbb{K}(x)$ and $\tau\in\mathrm{supp}(f)$, the \emph{Mahler dispersion} of $f$ at $\tau$, denoted by $\mathrm{disp}(f,\tau)$, is defined as follows.

If $\tau\in \mathcal{T}$, $\mathrm{disp}(f,\tau)$ is the largest $d\in\mathbb{Z}_{\geq 0}$ (if it exists) for which there exists $\alpha\in\mathrm{sing}(f,\tau)$ such that $\alpha^{p^d}\in\mathrm{sing}(f,\tau)$. If there is no such $d\in\mathbb{Z}_{\geq 0}$, then we set $\mathrm{disp}(f,\tau)=\infty$.

If $\tau=\infty$, let us write $f_{\infty}=\sum_{i=n}^Nc_ix^i\in\mathbb{K}[x,x^{-1}]$ with $c_nc_N\neq 0$.
\begin{itemize}
\item If $f_{\infty}=c_0\neq 0$ then we set $\mathrm{disp}(f,\infty)=0$; otherwise
\item $\mathrm{disp}(f,\infty)$ is the largest $d\in\mathbb{Z}_{\geq 0}$ for which there exists an index $i\neq 0$ such that $c_i\neq 0$ and $c_{ip^d}\neq 0$.
\end{itemize} 
For $f\in\mathbb{K}(x)$ and $\tau\in\mathcal{T}\cup\{\infty\}$ such that $\tau\notin\mathrm{supp}(f)$, we do not define $\mathrm{disp}(f,\tau)$ at all (cf.~\cite{Abramov:1971,Abramov:1974,ChenSinger2012}).
\end{defn}

Similarly as in the shift and $q$-difference cases (cf.~\cite[Lemma 6.3]{HardouinSinger2008} and \cite[Lemma~2.4 and Lemma~2.9]{ChenSinger2012}), Mahler dispersions will play a crucial role in what follows. As we prove in Theorem~\ref{thm:summable-dispersion}, they already provide a partial obstruction to summability: if $f\in\mathbb{K}(x)$ is $\lambda$-Mahler summable then almost every Mahler dispersion of $f$ is non-zero. Moreover, Mahler dispersions also detect whether $f$ has any ``bad'' poles (\ie, at roots of unity of order coprime to $p$) according to the following result proved in \cite[Lem.~2.16]{arreche-zhang:2022}.

\begin{lem}[\protect{\cite[Lem.~2.16]{arreche-zhang:2022}}] \label{lem:infinite-dispersion} Let $f\in\mathbb{K}(x)$ and $\tau\in\mathrm{supp}(f)$. Then $\mathrm{disp}(f,\tau)=\infty$ if and only if $\mathrm{sing}(f,\tau)\cap\mathcal{C}(\tau)\neq \emptyset$.
\end{lem}

\subsection{Mahler coefficients}\label{sec:mahler-coefficients}

Here we extend our study of the effect of the Mahler operator $\sigma$ on partial fraction decompositions initiated in \cite[\S2.4]{arreche-zhang:2022}. For $\alpha\in\mathbb{K}^\times$ and $m,k,n\in\mathbb{Z}$ with $n\geq 0$ and $1\leq k \leq m$, we define the \emph{Mahler coefficients} $V^m_{k,n}(\alpha)\in\mathbb{K}$ implicitly by\begin{equation}\label{eq:mahler-coefficients} \sigma^n\left(\frac{1}{(x-\alpha^{p^n})^m}\right)=\frac{1}{(x^{p^n}-\alpha^{p^n})^m}=\sum_{k=1}^m\sum_{i=0}^{p^n-1}\frac{V^m_{k,n}(\zeta_{p^n}^i\alpha)}{(x-\zeta_{p^n}^i\alpha)^k}.\end{equation} These Mahler coefficients are computed explicitly with the following result, proved analogously to the similar \cite[Lem.~2.17]{arreche-zhang:2022} in case $n=1$.

\begin{lem}\label{lem:mahler-coefficients} For every $\alpha\in \mathbb{K}^\times$, the Mahler coefficients \[V^m_{k,n}(\alpha)=\mathbb{V}^m_{k,n}\cdot\alpha^{k-mp^n},\] where the universal coefficients $\mathbb{V}^m_{k,n}\in\mathbb{Q}$ are the first $m$ Taylor coefficients at $x=1$ of \begin{equation}\label{eq:taylor-coefficients}(x^{p^n-1}+\dots+x+1)^{-m}=\sum_{k=1}^m\mathbb{V}^m_{k,n}\cdot(x-1)^{m-k}+O((x-1)^m).\end{equation}\end{lem}

Although Lemma~\ref{lem:mahler-coefficients} serves to compute the $V^m_{k,n}(\alpha)$ for $\alpha\in\mathbb{K}^\times$, $n\in\mathbb{Z}_{\geq 0}$, and $1\leq k \leq m$ efficiently in practice\footnote{That is, by computing successive derivatives of the left-hand side and evaluating at $x=1$.}, the following result provides an explicit symbolic expression for these Mahler coefficients as sums over certain integer partitions.

\begin{defn}\label{defn:partitions} For $k,n\in\mathbb{Z}_{\geq 0}$, let $\Pi_{n}(k)$ be the set of integer partitions $\mu=(\mu_1,\dots,\mu_\ell)$ of $k$ with greatest part $\mu_1<p^n$, and denote by $\ell(\mu):=\ell$ the length of $\mu$ and by $\ell_i(\mu)$ the multiplicity of $i$ in $\mu$ for $1\leq i\leq p^n-1$. We adopt the conventions that $\Pi_n(0)=\{\emptyset\}$ for every $n\geq 0$ and $\Pi_0(k)=\emptyset$ for every $k\geq 1$. The empty partition $\mu=\emptyset$ has length $\ell(\emptyset)=0$ and multiplicity $\ell_i(\emptyset)=0$ for every $1\leq i \leq p^n-1$ (vacuously so when $n=0$).
\end{defn}

\begin{prop}\label{prop:mahler-coefficients} For $n\geq 0$ and $1\leq k \leq m$, \[\mathbb{V}^m_{k,n}=p^{-nm}
\cdot\sum_{\mu\in\Pi_{n}(m-k)}(-p^n)^{-\ell(\mu)}\begin{pmatrix}m-1+\ell(\mu) \\ m-1,\ell_1(\mu),\dots,\ell_{p^n-1}(\mu)\end{pmatrix}\prod_{i=1}^{p^n-1}\begin{pmatrix}p^n \\ i+1\end{pmatrix}^{\ell_i(\mu)}.\]\end{prop}

\begin{proof} By Lemma~\ref{lem:mahler-coefficients}, $V^m_{k,n}(\alpha)=\mathbb{V}^m_{k,n}\cdot\alpha^{k-mp^n}$, where the $\mathbb{V}^m_{k,n}\in\mathbb{Q}$ are given by \eqref{eq:taylor-coefficients}. Writing $f_m(x)=x^{-m}$ and $g_n(x)=x^{p^n-1}+\dots+x+1$, and letting $W^m_{k,n}\in\mathbb{Q}$ be the coefficient of $(x-1)^k$ in the Taylor expansion of $(f_m\circ g_n)(x)$ at $x=1$ as in Lemma~\ref{lem:mahler-coefficients}, we have that $\mathbb{V}^m_{k,n}=W^m_{m-k,n}$ for every $1\leq k\leq m$. By Fa\`a di Bruno's formula \cite{johnson:2002}, we have \[W^m_{k,n}=\frac{(f_m\circ g_n)^{(k)}(1)}{k!}=\frac{1}{k!}\cdot\sum_{\mu\in\Pi(k)}\frac{k!}{\ell_1(\mu)!\cdots\ell_k(\mu)!}f^{(\ell(\mu))}_m(g_n(1))\prod_{i=1}^k\left(\frac{g_n^{(i)}(1)}{i!}\right)^{\ell_i(\mu)}\] for every $k\geq 0$, where $\Pi(k)$ denotes the set of \emph{all} partitions of $k$, and $\ell(\mu)$ and $\ell_i(\mu)$ are as in Definition~\ref{defn:partitions}. For every $\ell,i\in\mathbb{Z}_{\geq 0}$, we compute \[f_m^{(\ell)}(g_n(1))=(-1)^\ell p^{-n(m+\ell)}\frac{(m-1+\ell)!}{(m-1)!} \qquad \text{and} \qquad g_n^{(i)}(1)=i!\begin{pmatrix}p^n \\ i+1\end{pmatrix},\] where we adopt the usual convention that $\left(\begin{smallmatrix}p^n\\ i+1\end{smallmatrix}\right)=0$ whenever $i\geq p^n$. Therefore the partitions $\mu\in \Pi(k)\backslash \Pi_n(k)$ with greatest part $\mu_1\geq p^n$ do not contribute to the sum. \end{proof}

We isolate the following special case for ease of reference (cf.~\cite[Cor.~2.18]{arreche-zhang:2022}), since it arises often.

\begin{cor}\label{cor:mahler-coefficients} Let $\alpha\in\mathbb{K}^\times$, $m\in\mathbb{N}$, and $n\in\mathbb{Z}_{\geq 0}$. Then $V^m_{m,n}(\alpha)=p^{-nm}\alpha^{m-p^nm}$.\end{cor}

\begin{proof} In the special case where $k=m$ in Proposition~\ref{prop:mahler-coefficients}, the sum is over $\mu\in\Pi(0)=\{\emptyset\}$, and $\ell(\emptyset)=0=\ell_i(\emptyset)$ for every $i\in\mathbb{N}$, whence $V^m_{m,n}(\alpha)=p^{-nm}\alpha^{m-p^nm}$ by Lemma~\ref{lem:mahler-coefficients}.
\end{proof}

The Mahler coefficients $V^m_{k,n}(\alpha)$ defined above are the main ingredients in our definition of twisted Mahler discrete residues. Our proofs that these residues comprise a complete obstruction to $\lambda$-Mahler summability will rely on the following elementary computations, which we record here once and for all for future reference.

\begin{lem}\label{lem:mahler-coefficients-computations} Let $n\in\mathbb{Z}_{\geq 0}$, $\alpha\in\mathbb{K}^\times$, and $d_1,\dots,d_m\in\mathbb{K}$ for some $m\in \mathbb{N}$. Then \[\sigma^n\left(\sum_{k=1}^m\frac{d_k}{(x-\alpha^{p^n})^k}\right)=\sum_{k=1}^m\sum_{i=0}^{p^n-1}\frac{\sum_{s=k}^mV^s_{k,n}(\zeta_{p^n}^i\alpha)d_s}{(x-\zeta_{p^n}^i\alpha)^k}.\] For $\lambda\in\mathbb{Z}$ and $g\in\mathbb{K}(x)$, the element $\Delta_\lambda^{(n)}(g):=p^{\lambda n}\sigma^n(g)-g$ is $\lambda$-Mahler summable.
\end{lem}

\begin{proof} The claims are trivial if $n=0$: $\zeta_{1}=1$, $V^s_{k,0}(\alpha)=\delta_{s,k}$ (Kronecker's $\delta$) for $k\leq s \leq m$, and $\Delta^{(0)}_\lambda(g)=0$ is $\lambda$-Mahler summable. Suppose that $n\geq 1$. For $1\leq s\leq m$ we have \[\sigma^n\left(\frac{d_s}{(x-\alpha^{p^n})^s}\right)=\sum_{k=1}^s\sum_{i=0}^{p^n-1}\frac{V^s_{k,n}(\zeta_{p^n}^i\alpha)d_s}{(x-\zeta_{p^n}^i\alpha)^k}\] by definition (cf.~\eqref{eq:mahler-coefficients}), and it follows that\[\sigma^n\left(\sum_{s=1}^m\frac{d_s}{(x-\alpha^{p^n})^s}\right)=\sum_{s=1}^m\sum_{k=1}^s\sum_{i=0}^{p^n-1}\frac{V^s_{k,n}(\zeta_{p^n}^i\alpha)d_s}{(x-\zeta_{p^n}^i\alpha)^k}=\sum_{k=1}^m\sum_{i=0}^{p^n-1}\frac{\sum_{s=k}^mV^s_{k,n}(\zeta_{p^n}^i\alpha)d_s}{(x-\zeta_{p^n}^i\alpha)^k}.\] Finally, since \[\Delta^{(n)}_\lambda(g)=p^{\lambda n}\sigma^n(g)-g=p^\lambda\sigma\left(\sum_{j=0}^{n-1}p^{\lambda j}\sigma^j(g)\right)-\left(\sum_{j=0}^{n-1}p^{\lambda j}\sigma^j(g)\right)=\Delta_\lambda\left(\sum_{j=0}^{n-1}p^{\lambda j}\sigma^j(g)\right),\]  $\Delta_\lambda^{(n)}(g)$ is $\lambda$-Mahler summable. 
\end{proof}


\section{Cycle maps and their $\omega$-sections}\label{sec:cycle-maps} The goal of this section is to define and study the properties of two auxiliary maps $\mathcal{D}_{\lambda,\tau}$ and $\mathcal{I}^{(\omega)}_{\lambda,\tau}$ that will help us retain some control over the perverse periodic behavior of the roots of unity $\gamma\in\mathcal{C}(\tau)$ under the $p$-power map $\gamma\mapsto \gamma^p$. The following definitions and results are relevant only for torsion Mahler trees $\tau\in\mathcal{T}_+$.

\begin{defn}\label{defn:cyclic-component} With notation as in Definition~\ref{defn:cycles}, let $\tau\in\mathcal{T}_+$ be a torsion Mahler tree, let $g\in\mathbb{K}(x)$, and let us write the $\tau$-component $g_\tau$ of $g$ from Definition~\ref{defn:tree-projection} as \[g_\tau=\sum_{k\in\mathbb{N}}\sum_{\alpha\in\tau}\frac{d_k(\alpha)}{(x-\alpha)^k}.\] We define the \emph{cyclic component} of $g_\tau$ by \[\mathcal{C}(g_\tau):=\sum_{k\in\mathbb{N}}\sum_{\gamma\in\mathcal{C}(\tau)}\frac{d_k(\gamma)}{(x-\gamma)^k}.\]
\end{defn}

\begin{defn}\label{defn:d-map}Let $\mathcal{S}:=\bigoplus_{k\in\mathbb{N}}\mathbb{K}$ denote the $\mathbb{K}$-vector space of finitely supported sequences in $\mathbb{K}$. For $\tau\in\mathcal{T}_+$, we let $\mathcal{S}^{\mathcal{C}(\tau)}:=\bigoplus_{\gamma\in\mathcal{C}(\tau)}\mathcal{S}$. For $\lambda\in\mathbb{Z}$, we define \emph{cycle map} $\mathcal{D}_{\lambda,\tau}$ to be the $\mathbb{K}$-linear endomorphism \begin{equation}\label{eq:d-map-true}
\mathcal{D}_{\lambda,\tau}:  \mathcal{S}^{\mathcal{C}(\tau)} \rightarrow\mathcal{S}^{\mathcal{C}(\tau)}: 
\left(d_{k}(\gamma)\right)_{\substack{k\in\mathbb{N}\\ \gamma\in\mathcal{C}(\tau)} }\mapsto \left(-d_{k}(\gamma)+p^\lambda\sum_{s\geq k}V^s_{k,1}(\gamma)\cdot d_{s}(\gamma^p)\right)_{\substack{k\in\mathbb{N}\\ \gamma\in\mathcal{C}(\tau)} },\end{equation} where the Mahler coefficients $V^s_{k,1}(\gamma)$ are defined as in \eqref{eq:mahler-coefficients}.\end{defn}

We treat the $\mathbb{K}$-vector space $\mathcal{S}^{\mathcal{C}(\tau)}$ introduced in the preceding Definition~\ref{defn:d-map} as an abstract receptacle for the coefficients occurring in the partial fraction decomposition of $\mathcal{C}(g_\tau)$ for $\tau\in\mathcal{T}_+$ and arbitrary elements $g\in \mathbb{K}(x)$. Note that the infinite summation in \eqref{eq:d-map-true} is harmless, since $d_s(\gamma^p)=0$ for every $\gamma\in\mathcal{C}(\gamma)$ for large enough $s\in\mathbb{N}$. The cycle map $\mathcal{D}_{\lambda,\tau}$ for $\lambda=0$ is the negative of the (truncated) linear map introduced in~\cite[Lemma~4.14]{arreche-zhang:2022}. 
 The relevance of $\mathcal{D}_{\lambda,\tau}$ to our study of $\lambda$-Mahler summability is captured by the following immediate computational result.

\begin{lem}\label{lem:cyclic-component-short} Let $\lambda\in\mathbb{Z}$, $g\in\mathbb{K}(x)$, and $\tau\in\mathcal{T}_+$. Let us write the cyclic components \[\mathcal{C}(g_\tau)=\sum_{k\in\mathbb{N}}\sum_{\gamma\in\mathcal{C}(\tau)}\frac{d_k(\gamma)}{(x-\gamma)^k}\qquad\text{and}\qquad \mathcal{C}\left(\Delta_\lambda(g_\tau)\right)=\sum_{k\in\mathbb{N}}\sum_{\gamma\in\mathcal{C}(\tau)}\frac{c_k(\gamma)}{(x-\gamma)^k}\] as in Definition~\ref{defn:cyclic-component}. Writing $\mathbf{d}:=(d_k(\gamma))_{k,\gamma}$ and $\mathbf{c}:=(c_k(\gamma))_{k,\gamma}$ as vectors in $\mathcal{S}^{\mathcal{C}(\tau)}$ as in Definition~\ref{defn:d-map}, we have $\mathbf{c}=\mathcal{D}_{\lambda,\tau}(\mathbf{d})$.
\end{lem}

\begin{proof}It follows from Lemma~\ref{lem:mahler-coefficients-computations} that \[\mathcal{C}(\sigma(g_\tau))=\sum_{k\in\mathbb{N}}\sum_{\gamma\in\mathcal{C}(\tau)}\frac{\sum_{s\geq k} V^s_{k,1}(\gamma)d_s(\gamma^p)}{(x-\gamma)^k},\] and therefore, for every $k\in\mathbb{N}$ and $\gamma\in\mathcal{C}(\tau)$, \[c_k(\gamma)=-d_k(\gamma)+p^\lambda\sum_{s\geq k} V^s_{k,1}(\gamma)d_s(\gamma^p).\qedhere\] 
\end{proof}

The following fundamental Lemma is essential to our study of $\lambda$-Mahler summability at torsion Mahler trees $\tau\in\mathcal{T}_+$.

\begin{lem}\label{lem:d-map-kernel} Let $\lambda\in\mathbb{Z}$, $\tau\in\mathcal{T}_+$, and set  $e:=|\mathcal{C}(\tau)|$ as in Definition~\ref{defn:cycles}. Let $\mathcal{D}_{\lambda,\tau}$ be as in Definition~\ref{defn:d-map}. 
\begin{enumerate}
\item If $\lambda\leq 0$ then $\mathcal{D}_{\lambda,\tau}$ is an isomorphism.
\item If $\lambda\geq 1$ then $\mathrm{im}(\mathcal{D}_{\lambda,\tau})$ has codimension $1$ in $\mathcal{S}^{\mathcal{C}(\tau)}$ and $\mathrm{ker}(\mathcal{D}_{\lambda,\tau})=\mathbb{K}\cdot \mathbf{w}^{(\lambda)}$, where the vector $(w^{(\lambda)}_k(\gamma))=\mathbf{w}^{(\lambda)}\in\mathcal{S}^{\mathcal{C}(\tau)}$ is recursively determined by the conditions 

\begin{equation}\label{eq:kernel-recursion}
w^{(\lambda)}_k(\gamma):=
\begin{cases}0 & \text{for}\ k>\lambda;\\
\gamma^\lambda & \text{for}\ k=\lambda; \\
\displaystyle\frac{p^\lambda\gamma^k}{1-p^{(\lambda-k)e}}\sum_{j=0}^{e-1}\sum_{s=k+1}^\lambda p^{(\lambda-k)j}\mathbb{V}^s_{k,1}\gamma^{-sp^{j+1}} w^{(\lambda)}_s\bigl(\gamma^{p^{j+1}}\bigr) & \text{for any remaining} \ k<\lambda;
\end{cases}
    \end{equation}
    for each $\gamma\in\mathcal{C}(\tau)$, where the universal Mahler coefficients $\mathbb{V}^s_{k,1}\in\mathbb{Q}$ are as in Proposition~\ref{prop:mahler-coefficients}.
\end{enumerate}
\end{lem}
\begin{proof} Let $(d_k(\gamma))=\mathbf{d}\in\mathcal{S}^{\mathcal{C}(\tau)}-\{\mathbf{0}\}$, let $m\in\mathbb{N}$ be as large as possible such that $d_m(\gamma)\neq 0$ for some $\gamma\in\mathcal{C}(\tau)$, and let us write $(c_k(\gamma))=\mathbf{c}:=\mathcal{D}_{\lambda,\tau}(\mathbf{d})$.

Let us first assume that $\mathbf{d}\in\mathrm{ker}(\mathcal{D}_{\lambda,\tau})\Leftrightarrow\mathbf{c}=\mathbf{0}$. Then by the Definition~\ref{defn:d-map} and our choice of $m$, for each $\gamma\in\mathcal{C}(\tau)$, \begin{equation}\label{eq:d-kernel} 0=c_m(\gamma)
=p^\lambda V^m_{m,1}(\gamma)d_m(\gamma^p)-d_m(\gamma)=p^{\lambda-m}\gamma^{m-pm}d_m(\gamma^p)-d_m(\gamma),\end{equation} where the second equality results from Corollary~\ref{cor:mahler-coefficients}. Since \eqref{eq:d-kernel} holds for every $\gamma\in\mathcal{C}(\tau)$ simultaneously, it follows that $d_m(\gamma^{p^{j+1}})=p^{m-\lambda}\gamma^{(p^{j+1}-p^j)m}d_m(\gamma^{p^j})$ for every $j\geq 0$ and for each $\gamma\in\mathcal{C}(\gamma)$, whence none of the $d_m(\gamma^{p^j})$ can be zero. Since $\gamma^{p^e}=\gamma$, we find that \[1=\frac{d_m(\gamma^{p^e})}{d_m(\gamma)}=\prod_{j=0}^{e-1}\frac{d_m(\gamma^{p^{j+1}})}{d_m(\gamma^{p^j})}=\prod_{j=0}^{e-1}p^{m-\lambda}\gamma^{(p^{j+1}-p^j)m}=p^{(m-\lambda)e}\gamma^{(p^e-1)m}=p^{(m-\lambda)e},\] which is only possible if $m=\lambda$. Therefore $d_k(\gamma)=0$ for every $k>\lambda$, whence $\mathcal{D}_{\lambda,\tau}$ is injective in case $\lambda\leq 0$. In case $\lambda\geq 1$, it also follows from \eqref{eq:d-kernel} with $m=\lambda$ that $\gamma^{-p\lambda}d_\lambda(\gamma^p)=\gamma^{-\lambda} d_\lambda(\gamma)=\omega$ must be a constant that does not depend on $\gamma\in\mathcal{C}(
\gamma)$. We claim that if we further impose that this $\omega=1,$ then the remaining componenets of our vector $\mathbf{d}$ are uniquely determined by the recursion \eqref{eq:kernel-recursion}. Indeed, if $\lambda=1$ then there are no more components to determine, whereas if $\lambda\geq 2$ then we must have, for $1\leq k \leq \lambda-1$, \begin{gather*}0=-d_k(\gamma)+p^\lambda\sum_{s=k}^\lambda V^s_{k,1}(\gamma)d_s(\gamma^p)\qquad\Longleftrightarrow \\ d_k(\gamma)-p^{\lambda-k}\gamma^{k-pk}d_k(\gamma^p)=d_k(\gamma)-p^\lambda V^k_{k,1}(\gamma)d_k(\gamma^p)=p^\lambda\sum_{s=k+1}^\lambda V^s_{k,1}(\gamma)d_s(\gamma^p),
\end{gather*} where the first equality is obtained from Corollary~\ref{cor:mahler-coefficients} and the second is just a rearrangement. Replacing the arbitrary $\gamma$ in the above equation with $\gamma^{p^j}$ for $j=0,\dots,e-1$, we find that the telescoping sum \begin{multline*}\gamma^{-k}\bigl(1-p^{(\lambda-k)e}\bigr)d_k(\gamma)=
\sum_{j=0}^{e-1}p^{(\lambda-k)j}\gamma^{-kp^j}\cdot\Bigl(d_k\bigl(\gamma^{p^j}\bigr)-p^{\lambda-k}\gamma^{kp^j-kp^{j+1}} d_k\bigl(\gamma^{p^{j+1}}\bigr)\Bigr)\\
=\sum_{j=0}^{e-1}p^{(\lambda-k)j}\gamma^{-kp^j}\cdot p^\lambda\sum_{s=k+1}^\lambda V^s_{k,1}\bigl(\gamma^{p^j}\bigr) d_s\bigl(\gamma^{p^{j+1}}\bigr)=p^\lambda\sum_{j=0}^{e-1}\sum_{s=k+1}^\lambda p^{(\lambda-k)j}\mathbb{V}^s_{k,1}\gamma^{-sp^{j+1}} d_s\bigl(\gamma^{p^{j+1}}\bigr),\end{multline*} which is clearly equivalent to the expression defining the components $w^{(\lambda)}_k(\gamma)$ for $k<\lambda$ in \eqref{eq:kernel-recursion}, and where we have once again used Lemma~\ref{lem:mahler-coefficients} to obtain the last equality, since $V^s_{k,1}(\gamma^{p^j})=\mathbb{V}^s_{k,1}\gamma^{kp^j-sp^{j+1}}$. This concludes the proof of the statements concerning $\mathrm{ker}(\mathcal{D}_{\lambda,\tau})$.

Let us now prove the statements concerning $\mathrm{im}(\mathcal{D}_{\lambda,\tau})$. We see from Definition~\ref{defn:d-map} that $\mathcal{D}_{\lambda,\tau}$ preserves the increasing filtration of $\mathcal{S}^{\mathcal{C}(\tau)}$ by the finite-dimensional subspaces \begin{equation}\label{eq:filtration}
\mathcal{S}^{\mathcal{C}(\tau)}_{<m}:=\left\{(d_k(\gamma))\in\mathcal{S}^{\mathcal{C}(\tau)} \ \middle| \ d_k(\gamma)=0 \ \text{for} \ k\geq m \ \text{and every} \ \gamma\in\mathcal{C}(\tau)\right\}.\end{equation} In case $\lambda\leq 0$, since $\mathcal{D}_{\lambda,\tau}$ is injective, it must restrict to an automorphism of $\mathcal{S}^{\mathcal{C}(\tau)}_{<m}$ for each $m\in\mathbb{N}$, concluding the proof of (1). In case $\lambda\geq 1$, since $\mathrm{ker}(\mathcal{D}_{\lambda,\tau})$ is one dimensional, it follows that $\mathrm{im}(\mathcal{D}_{\lambda,\tau})\cap \mathcal{S}^{\mathcal{C}(\tau)}_{<m}$ has codimension $1$ in $\mathcal{S}^{\mathcal{C}(\tau)}_{<m}$ for every $m\geq \lambda+1$, and therefore $\mathrm{im}(\mathcal{D}_{\lambda,\tau})$ has codimension $1$ in all of $\mathcal{S}^{\mathcal{C}(\tau)}$. This concludes the proof.\end{proof}

\begin{defn}\label{defn:omega-section} Let $\lambda\in\mathbb{Z}$, $\tau\in\mathcal{T}_+$, and set $e:=|\mathcal{C}(\tau)|$ as in Definition~\ref{defn:cycles}. We define the \mbox{$0$\emph{-section}} $\mathcal{I}^{(0)}_{\lambda,\tau}$ (of the map $\mathcal{D}_{\lambda,\tau}$ of Definition~\ref{defn:d-map}) as follows. For $(c_k(\gamma))=\mathbf{c}\in\mathcal{S}^{\mathcal{C}(\tau)}$, let us write  $(d_k(\gamma))=\mathbf{d}=\mathcal{I}^{(0)}_{\lambda,\tau}(\mathbf{c})\in\mathcal{S}^{\mathcal{C}(\tau)}$. 
 We set each $d_k(\gamma)=0$ whenever $k\in\mathbb{N}$ is such that $c_k(\gamma)=0$ for every $\gamma\in\mathcal{C}(\tau)$. 
For any remaining $k\in\mathbb{N}$, we define recursively \begin{gather}\label{eq:omega-section-k}d_k(\gamma):=
   \frac{\gamma^k}{p^{(\lambda-k)e}-1}\sum_{j=0}^{e-1}p^{(\lambda-k)j}\gamma^{-kp^j}\left[ c_k\bigl(\gamma^{p^j}\bigr)-p^\lambda\sum_{s\geq k+1}V^s_{k,1} \bigl(\gamma^{p^j}\bigr) d_s \bigl(\gamma^{p^{j+1}}\bigr)     \right] \qquad \text{for} \ k\neq \lambda;\intertext{and, if $\lambda\geq 1$, we set}
d_\lambda(\gamma):=\frac{\gamma^\lambda}{e}\sum_{j=0}^{e-1}(j+1-e)\gamma^{-\lambda p^{j}}\left[c_\lambda\bigl(\gamma^{p^{j}}\bigr)-p^\lambda\sum_{s\geq\lambda+1} V^s_{\lambda,1}\bigl(\gamma^{p^{j}}\bigr)d_s\bigl(\gamma^{p^{j+1}}\bigr)\right].   \label{eq:omega-section-lambda}
\end{gather}
    More generally, for any $\omega\in\mathbb{K}$ the $\omega$\emph{-section} $\mathcal{I}^{(\omega)}_{\lambda,\tau}$ (of $\mathcal{D}_{\lambda,\tau}$) is defined by setting  \begin{equation}\label{eq:omega-section}
        \mathcal{I}^{(\omega)}_{\lambda,\tau}(\mathbf{c}):=\begin{cases}\mathcal{I}^{(0)}_{\lambda,\tau}(\mathbf{c}) & \text{if} \ \lambda\leq 0; \\ \vphantom{\displaystyle \sum^\lambda}\mathcal{I}^{(0)}_{\lambda,\tau}(\mathbf{c})+\omega\mathbf{w}^{(\lambda)} & \text{if} \ \lambda\geq 1;
        \end{cases}
    \end{equation} for every $\mathbf{c}\in\mathcal{S}^{\mathcal{C}(\tau)},$ where $\mathbf{w}^{(\lambda)}$ is the vector defined in \eqref{eq:kernel-recursion} for $\lambda\geq 1$.
\end{defn}

\begin{prop}\label{prop:omega-section} Let $\lambda\in\mathbb{Z}$, $\tau\in\mathcal{T}_+$, and set $e:=|\mathcal{C}(\tau)|$ as in Definition~\ref{defn:cycles}. Let $\omega\in\mathbb{K}$ and consider the $\omega$-section $\mathcal{I}^{(\omega)}_{\lambda,\tau}$ of Definition~\ref{defn:omega-section}. Let $\mathbf{c}\in\mathcal{S}^{\mathcal{C}(\tau)}$, and let us write $\mathbf{d}:=\mathcal{I}^{(\omega)}_{\lambda,\tau}(\mathbf{c})$ and $\tilde{\mathbf{c}}:=\mathcal{D}_{\lambda,\tau}(\mathbf{d})$ as in Definition~\ref{defn:d-map}. Then 
\begin{equation}\label{eq:omega-section-i}
c_k(\gamma)=\tilde{c}_k(\gamma)  \qquad \text{whenever} \ k\neq \lambda, \ \text{for every} \ \gamma\in\mathcal{C}(\gamma);\end{equation} and, in case $\lambda\geq 1$,  
\begin{equation}
\label{eq:omega-section-c}
    c_\lambda(\gamma)-\tilde{c}_\lambda(\gamma)=\frac{\gamma^\lambda}{e}\sum_{j=1}^e \gamma^{-\lambda p^j}\left(c_\lambda\Bigl(\gamma^{p^j}\Bigr)-p^\lambda\sum_{s\geq\lambda+1}V^s_{\lambda,1}\Bigl(\gamma^{p^j}\Bigr)d_s\Bigl(\gamma^{p^{j+1}}\Bigr)\right).
\end{equation} Moreover, $\mathbf{c}\in\mathrm{im}(\mathcal{D}_{\lambda,\tau})$ if and only if $\mathbf{c}=\tilde{\mathbf{c}}$.  \end{prop}

\begin{proof}
    The expression~\eqref{eq:omega-section-k} arises from a similar computation as in the proof of Lemma~\ref{lem:d-map-kernel}. Let $\mathbf{c}\in\mathcal{S}^{\mathcal{C}(\tau)}$ be arbitrary, and let us try (and maybe fail), to construct $\mathbf{d}\in\mathcal{S}^{\mathcal{C}(\tau)}$ such that $\mathcal{D}_{\lambda,\tau}(\mathbf{d})=\mathbf{c}$, that is, with \begin{gather}\label{eq:omega-section-proof} c_k(\gamma)=-d_k(\gamma)+p^\lambda\sum_{s\geq k} V^s_{k,1} (\gamma)d_s(\gamma) \qquad \Longleftrightarrow \\ 
    p^{\lambda-k}\gamma^{k-pk}d_k(\gamma^p)-d_k(\gamma)=c_k(\gamma)-p^\lambda\sum_{s\geq k+1}V^s(\gamma)d_s(\gamma^p).
    \end{gather} Then we again have the telescoping sum \begin{gather*}\bigl(p^{(\lambda-k)e}-1\bigr)\gamma^{-k}d_k(\gamma)=\sum_{j=0}^{e-1}p^{(\lambda-k)j}\gamma^{-kp^j}\cdot\Bigl(p^{\lambda-k}\gamma^{kp^j-kp^{j+1}}d_k\bigl(\gamma^{p^{j+1}}\bigr)-d_k\bigl(\gamma^{p^j}\bigr)\Bigr) \\ =\sum_{j=0}^{e-1}p^{(\lambda-k)j}\gamma^{-kp^j}\cdot \left(c_k\bigl(\gamma^{p^j}\bigr)-p^\lambda\sum_{s\geq k+1}V^s_{k,1}\bigl(\gamma^{p^j}\bigr)d_s\bigl(\gamma^p\bigr)\right),\end{gather*} which is equivalent to \eqref{eq:omega-section-k} provided precisely that $k\neq \lambda$. Thus we see that \eqref{eq:omega-section-k} is a \emph{necessary} condition on the $d_k(\gamma)$ in order to satisfy \eqref{eq:omega-section-i}. In case $\lambda\leq 0$, we know that $\mathcal{D}_{\lambda,\tau}$ is an isomorphism by Lemma~\ref{lem:d-map-kernel}, in which case this condition must also be sufficient and we have nothing more to show. 
    
    Let us assume from now on that $\lambda\geq 1$. Since by Lemma~\ref{lem:d-map-kernel} the restriction of $\mathcal{D}_{\lambda,\tau}$ to \begin{equation*}
       \mathcal{S}^{\mathcal{C}(\tau)}_{>\lambda} :=\left\{\mathbf{d}\in\mathcal{S}^{\mathcal{C}(\tau)} \ \middle| \ d_k(\gamma)=0 \ \text{for every} \ k\leq\lambda \ \text{and} \ \gamma\in\mathcal{C}(\gamma)\right\}\end{equation*} is injective, and since it preserves the induced filtration \eqref{eq:filtration}, it follows that $\mathrm{pr}_\lambda\circ\mathcal{D}_{\lambda,\tau}$ restricts to an automorphism of $\mathcal{S}^{\mathcal{C}(\tau)}_{>\lambda}$, where $\mathrm{pr}_\lambda:\mathcal{S}^{\mathcal{C}(\tau)}\twoheadrightarrow\mathcal{S}^{\mathcal{C}(\tau)}_{\>\lambda}$ denotes the obvious projection map. Therefore the necessary condition \eqref{eq:omega-section-k} must also be sufficient in order to satisfy \eqref{eq:omega-section-i} for $k>\lambda$. Since $\mathcal{D}_{\lambda,\tau}$ also restricts to an automorphism of $\mathcal{S}^{\mathcal{C}(\tau)}_{<\lambda}$ (trivially so in case $\lambda=1$, since $\mathcal{S}^{\mathcal{C}(\tau)}_{<1}=\{\mathbf{0}\}$), it similarly follows that the necessary condition \eqref{eq:omega-section} must also be sufficient in order to satisfy \eqref{eq:omega-section-i} for any $k<\lambda$ also, regardless of how the $d_\lambda(\gamma)$ are chosen.
       
       Now for the prescribed choice of $d_\lambda(\gamma)$ in \eqref{eq:omega-section-lambda}, we compute \begin{equation}\label{eq:omega-section-proof-1}\tilde{c}_\lambda(\gamma)-p^\lambda\sum_{s\geq \lambda+1}V^s_{\lambda,1}(\gamma)d_s(\gamma^p)=p^\lambda V^\lambda_{\lambda,1}(\gamma)d_\lambda(\gamma^p)-d_\lambda(\gamma)=\gamma^{\lambda-p\lambda}d_\lambda(\gamma^p)-d_\lambda(\gamma),\end{equation}

where the first equality follows from the definition of $\tilde{\mathbf{c}}=\mathcal{D}_{\lambda,\tau}(\mathbf{d})$, and the second equality from Corollary~\ref{cor:mahler-coefficients}. On the other hand, after re-indexing the sum in \eqref{eq:omega-section-lambda}, evaluated at $\gamma^p$ instead of $\gamma$, we find that \[\gamma^{\lambda-p\lambda}d_\lambda(\omega^p)=\frac{\gamma^\lambda}{e}\sum_{j=1}^e(j-e)\gamma^{-\lambda p^j}\left[c_\lambda\bigl(\gamma^{p^j}\bigr)-p^\lambda\sum_{s\geq \lambda+1}V^s_{\lambda,1}\bigl(\gamma^{p^j}\bigr) d_s\bigl(\gamma^{p^{j+1}}\bigr)\right],\] and after subtracting $d_\lambda(\gamma)$ exactly as given in \eqref{eq:omega-section-lambda} we find that \begin{multline}\label{eq:omega-section-proof-2}\gamma^{\lambda-p\lambda}d_\lambda(\gamma^p)-d_\lambda(\gamma)=-\frac{\gamma^\lambda}{e}\sum_{j=1}^{e-1}\gamma^{-\lambda p^j}\left[c_\lambda\bigl(\gamma^{p^j}\bigr)-p^\lambda\sum_{s\geq \lambda+1}V^s_{\lambda,1}\bigl(\gamma^{p^j}\bigr) d_s\bigl(\gamma^{p^{j+1}}\bigr)\right]\\-\frac{\gamma^\lambda}{e}(1-e)\gamma^{-\lambda}\left[c_\lambda\bigl(\gamma\bigr)-p^\lambda\sum_{s\geq \lambda+1}V^s_{\lambda,1}(\gamma) d_s\bigl(\gamma^p\bigr)\right]\\ = -\frac{\gamma^\lambda}{e}\sum_{j=0}^{e-1}\gamma^{-\lambda p^j}\left[c_\lambda\bigl(\gamma^{p^j}\bigr)-p^\lambda\sum_{s\geq \lambda+1}V^s_{\lambda,1}\bigl(\gamma^{p^j}\bigr) d_s\bigl(\gamma^{p^{j+1}}\bigr)\right]+c_\lambda(\gamma)-p^\lambda\sum_{s\geq \lambda+1} V^s_{\lambda,1}(\gamma)d_s(\gamma^p),\end{multline} with the convention that the sum $\sum_{j=1}^{e-1}$ is empty in case $e=1$. Putting \eqref{eq:omega-section-proof-1} and \eqref{eq:omega-section-proof-2} together establishes \eqref{eq:omega-section-c}. Since $\mathbf{c}=\tilde{\mathbf{c}}$ is a non-trivial sufficient for $\mathbf{c}\in\mathrm{im}(\mathcal{D}_{\lambda,\tau})$, by Lemma~\ref{lem:d-map-kernel} it must also be necessary, since $\mathrm{im}(\mathcal{D}_{\lambda,\tau})$ has codimension $1$ in $\mathcal{S}^{\mathcal{C}(\tau)}$. This concludes the proof. \end{proof}


\section{Mahler dispersion and $\lambda$-Mahler summability}\label{sec:summable-dispersion}

Our goal in this section is to prove Theorem~\ref{thm:summable-dispersion}: if $f\in\mathbb{K}(x)$ is $\lambda$-Mahler summable for some $\lambda\in\mathbb{Z}$, then it has non-zero dispersion almost everywhere, generalizing to arbitrary $\lambda\in\mathbb{Z}$ the analogous result for $\lambda=0$ obtained in \cite[Corollary~3.2]{arreche-zhang:2022}. In spite of the exceptions that occur for $\lambda\geq 1$, this will be an essential tool in our proofs that twisted Mahler discrete residues comprise a complete obstruction to $\lambda$-Mahler summability. 

In the following preliminary result, which generalizes \cite[Proposition~3.1]{arreche-zhang:2022} from the special case $\lambda=0$ to arbitrary $\lambda\in\mathbb{Z}$, we relate the Mahler dispersions of a $\lambda$-Mahler summable $f\in\mathbb{K}(x)$ to those of a certificate $g\in\mathbb{K}(x)$ such that $f=\Delta_\lambda(g)$.

\begin{prop}\label{prop:summable-dispersion} Let $f,g\in\mathbb{K}(x)$ and $\lambda\in\mathbb{Z}$ such that $f=\Delta_\lambda(g)$.

\begin{enumerate} 
\item If $\infty\in\mathrm{supp}(f)$, then $\mathrm{disp}(f,\infty)=\mathrm{disp}(g,\infty)+1$, except in case $\lambda\neq 0$ and the Laurent polynomial component $f_{\infty}=c_0\in\mathbb{K}^\times$, in which case we must have \mbox{$g_{\infty}=c_0/(p^\lambda-1)$}.
\item If $\infty\neq\tau\in\mathrm{supp}(f)$, then $\mathrm{disp}(f,\tau)=\mathrm{disp}(g,\tau)+1$, with the convention that $\infty+1=\infty$, except possibly in case that: $\mathcal{C}(\tau)$ is non-empty; and $\lambda\geq 1$; and the order of every pole of $g$ in $\mathcal{C}(\tau)$ is exactly~$\lambda$.
\end{enumerate} 
\end{prop}

\begin{proof} (1).~First suppose that $\{0\}\neq\theta\in\mathbb{Z}/\mathcal{P}$ is such that $g_\theta\neq 0$, and let us write \[g_\theta=\sum_{j=0}^dc_{ip^j}x^{ip^j},\] where we assume that $c_ic_{ip^d}\neq 0$, \ie, that $\mathrm{disp}(g_\theta,\infty)=d$. Then \[\Delta_\lambda(g_\theta)=p^\lambda c_{ip^d}x^{ip^{d+1}}-c_ix^{i}+\sum_{j=1}^d(p^\lambda c_{ip^{j-1}}-c_{ip^j})x^{ip^j},\] from which it follows that $0\neq f_\theta=\Delta_\lambda(g_\theta)$ and $\mathrm{disp}(f_\theta,\infty)=\mathrm{disp}(\Delta_\lambda(g_\theta),\infty)=d+1$. Since in this case Definition~\ref{defn:dispersion} gives that \[\mathrm{disp}(f,\infty)=\mathrm{max}\left\{\mathrm{disp}\left(f_\theta,\infty\right) \ \middle| \ \{0\}\neq\theta\in\mathbb{Z}/\mathcal{P}, f_\theta\neq 0\right\},\] and similarly for $\mathrm{disp}(g,\infty)$, we find that $\mathrm{disp}(f,\infty)=\mathrm{disp}(g,\infty)+1$ provided that the Laurent component $g_{\infty}\in\mathbb{K}[x,x^{-1}]$ is not constant.

In any case, by Lemma~\ref{lem:stable-support}, if $\infty\in\mathrm{supp}(f)$ then $\infty\in\mathrm{supp}(g)$. In this case, we have $0\neq f_{\infty}=\Delta_\lambda(g_{\infty})$, since $\infty\in\mathrm{supp}(f)$, and if $\lambda=0$ it follows in particular $g_{\infty}\notin\mathbb{K}$. In case $\lambda\neq 0$ and $f_{\infty}=c_0\in\mathbb{K}^\times$, the computation above shows that $g_\theta=0$ for every $\{0\}\neq \theta\in\mathbb{Z}/\mathcal{P}$, and we see that $g_{\infty}=g_{\{0\}}=c_0/(p^\lambda-1)$. 

(2).~Suppose $\tau\in\mathrm{supp}(f)$, and therefore $\tau\in\mathrm{supp}(g)$ by Lemma~\ref{lem:stable-support}. We consider two cases, depending on whether $\mathrm{disp}(g,\tau)$ is finite or not.

If $\mathrm{disp}(g,\tau)=:d<\infty$, let $\alpha\in \tau$ be such that $\alpha$ and $\alpha^{p^d}$ are poles of $g$. Choose $\gamma\in \tau$ such that $\gamma^p=\alpha$. Then $\gamma$ is a pole of $\sigma(g)$ but not of $g$ (by the maximality of $d$), and therefore $\gamma$ is a pole of $f$. On the other hand, $\gamma^{p^{d+1}}=\alpha^{p^d}$ is a pole of $g$ but not of $\sigma(g)$, for if $\alpha^{p^d}$ were a pole of $\sigma(g)$ then $\alpha^{p^{d+1}}$ would be a pole of $g$, contradicting the maximality of $d$. Therefore $\gamma^{p^{d+1}}$ is a pole of $f$. It follows that $\mathrm{disp}(f,\tau)\geq d+1$. One can show equality by contradiction: if $\alpha\in\tau$ is a pole of $f$ such that $\alpha^{p^s}$ is also a pole of $f$ for some $s> d+1$, then each of $\alpha$ and $\alpha^{p^s}$ is either a pole of $g$ or a pole of $\sigma(g)$. If $\alpha^{p^s}$ is a pole of $g$, then $\alpha$ cannot also be a pole of $g$, for this would contradict the maximality of $d$, whence $\alpha$ must be a pole of $\sigma(g)$, but then $\alpha^p$ would have to be a pole of $g$, still contradicting the maximality of $d$. Hence $\alpha^{p^s}$ must be a pole of $\sigma(g)$. But then $\alpha^{p^{s+1}}$ is a pole of $g$, which again contradicts the maximality of $d$ whether $\alpha$ is a pole of $\sigma(g)$ or of $g$. This concludes the proof that $\mathrm{disp}(f,\tau)=\mathrm{disp}(g,\tau)+1$ in this case where $\mathrm{disp}(g,\tau)<\infty$.

If $\mathrm{disp}(g,\tau)=\infty$ then $g$ has a pole in $\mathcal{C}(\tau)$ by Lemma~\ref{lem:infinite-dispersion}. If $f$ also has a pole in $\mathcal{C}(\tau)$ then $\mathrm{disp}(f,\tau)=\infty=\mathrm{disp}(g,\tau)+1$ and we are done. So let us suppose $\mathrm{disp}(f,\tau)<\infty$ and conclude that $g$ has a pole of order exactly $\lambda$ at every $\gamma\in\mathcal{C}(\tau)$. In this case, writing \[0\neq\mathcal{C}(g_\tau)=\sum_{k\in\mathbb{N}}\sum_{\gamma\in\mathcal{C}(\tau)}\frac{d_k(\gamma)}{(x-\gamma)^k}\qquad\text{and}\qquad 0=\mathcal{C}(f_\tau)=\sum_{k\in\mathbb{N}}\sum_{\gamma\in\mathcal{C}(\tau)}\frac{c_k(\gamma)}{(x-\gamma)^k}\] as in Definition~\ref{defn:cyclic-component}, it follows from Lemma~\ref{lem:cyclic-component-short} that $\mathcal{D}_{\lambda,\tau}(\mathbf{d})=\mathbf{c}$, where $\mathbf{d}:=(d_k(\gamma))$ and $\mathbf{c}:=(c_k(\gamma))$. By Lemma~\ref{lem:d-map-kernel}, $\mathbf{d}=\omega\mathbf{w}^{(\lambda)}$ for some $0\neq\omega\in\mathbb{K}$, where $\mathbf{w}^{(\lambda)}=(w^{(\lambda)}_k(\gamma)$ is the unique vector specified in Lemma~\ref{lem:d-map-kernel}, which has every component $w^{(\lambda)}_k(\gamma)=0$ for $k>\lambda$ and each component $w^{(\lambda)}_\lambda(\gamma)=\gamma^\lambda\neq 0$ for $\gamma\in\mathcal{C}(\tau)$.\end{proof}

In the next result we deduce from Proposition~\ref{prop:summable-dispersion} that if $f\in\mathbb{K}(x)$ is $\lambda$-Mahler summable then $f$ has non-zero dispersion almost everywhere. For the applications in the sequel, it will be essential for us to have these restrictions be defined intrinsically in terms of $f$, with no regard to any particular choice of certificate $g\in\mathbb{K}(x)$ such that $f=\Delta_\lambda(g)$.

\begin{thm} \label{thm:summable-dispersion} Let $\lambda\in\mathbb{Z}$ and and suppose that $f\in \mathbb{K}(x)$ is $\lambda$-Mahler summable. 
\begin{enumerate}
\item If $\infty\in\mathrm{supp}(f)$ and either $\lambda=0$ or $f_{\infty}\notin \mathbb{K}$ then ${\mathrm{disp}(f,\infty)>0}$. 
\item If $\lambda\leq 0$ then $\mathrm{disp}(f,\tau)>0$ for every $\infty\neq\tau\in\mathrm{supp}(f)$.
\item If $\lambda\geq 1$ and $\infty\neq\tau\in\mathrm{supp}(f)$ is such that either $\tau\in\mathcal{T}_0$ or $\mathrm{ord}(f,\tau)\neq \lambda$ then $\mathrm{disp}(f,\tau)>0$.
\end{enumerate}
\end{thm}

\begin{proof}
Suppose $f\in\mathbb{K}(x)$ is $\lambda$-Mahler summable and let $g\in\mathbb{K}(x)$ such that $f=\Delta_\lambda(g)$.

(1) and (2).~If $\infty\in\mathrm{supp}(f)$ then by Proposition~\ref{prop:summable-dispersion} $\mathrm{disp}(f,\infty)=\mathrm{disp}(g,\infty)+1>0$ provided that either $\lambda=0$ or $f_{\infty}\notin\mathbb{K}$. If $\lambda\leq 0$ then $\mathrm{disp}(f,\tau)=\mathrm{disp}(g,\tau)+1>0$ for all $\infty\neq\tau\in\mathrm{supp}(f)$ by Proposition~\ref{prop:summable-dispersion}. 

(3).~Assuming that $\lambda\geq 1$, we know by Proposition~\ref{prop:summable-dispersion} that $\mathrm{disp}(f,\tau)=\mathrm{disp}(g,\tau)+1>0$ for every $\infty\neq\tau\in\mathrm{supp}(f)$, except possibly in case $\tau\in\mathcal{T}_+$ and every pole of $g$ in $\mathcal{C}(\tau)$ has order exactly $\lambda$. Thus our claim is already proved for $\tau\in\mathcal{T}_0$. So from now on we suppose $\tau\in\mathcal{T}_+$. By Lemma~\ref{lem:stable-support}(7), $\mathrm{ord}(f,\tau)=\mathrm{ord}(g,\tau)$, and therefore if $\mathrm{ord}(f,\tau)<\lambda$, there are no poles of $g$ of order $\lambda$ anywhere in $\tau$, let alone in $\mathcal{C}(\tau)$, so $\mathrm{disp}(f,\tau)=\mathrm{disp}(g,\tau)+1>0$ by Proposition~\ref{prop:summable-dispersion} in this case also. Moreover, if $f$ has a pole of any order in $\mathcal{C}(\tau)$, then $\mathrm{disp}(f,\tau)=\infty>0$ by Lemma~\ref{lem:infinite-dispersion}. It remains to show that if $m:=\mathrm{ord}(f,\tau)>\lambda$ then $\mathrm{disp}(f,\tau)>0$. In this case, even though $\mathrm{ord}(g,\tau)=m>\lambda$ by Lemma~\ref{lem:stable-support} it may still happen that $g$ has a pole of order exactly $\lambda$ at every $\gamma\in\mathcal{C}(\tau)$ and yet the higher-order poles of $g$ lie in the complement $\tau-\mathcal{C}(\tau)$, in which case Proposition~\ref{prop:summable-dispersion} remains silent. So let $\alpha_1,\dots,\alpha_s\in\mathrm{sing}(g,\tau)$ be all the pairwise-distinct elements at which $g$ has a pole of order $m>\lambda$. 
Choose $\beta_j\in\tau$ such that $\beta_j^p=\alpha_j$ for each $j=1,\dots,s$, and let us write
\begin{gather*}
g_\tau=\sum_{j=1}^s\frac{d_j}{(x-\alpha_j)^m}+(\text{lower-order terms}),\quad \text{so that}\\
f_\tau = \sum_{j=1}^s\left(\sum_{i=0}^{p-1}\frac{p^\lambda V^m_{m,1}(\zeta_p^i\beta_j)\cdot d_j}{(x-\zeta_p^i\beta_j)^m}-\frac{d_j}{(x-\alpha_j)^m}\right)+(\text{lower-order-terms})
\end{gather*}
by Lemma~\ref{lem:mahler-coefficients-computations}. If any $\alpha_j\in\mathcal{C}(\tau)$, then we already have $\mathrm{disp}(f,\tau)=\mathrm{disp}(g,\tau)+1>0$ by Proposition~\ref{prop:summable-dispersion}. So we can assume without loss of generality that no $\alpha_j$ belongs to $\mathcal{C}(\tau)$, which implies that the $(p+1)\cdot s$ apparent poles $\zeta_p^i\beta_j$ and $\alpha_j$ of $f_\tau$ of order $m$ are pairwise distinct, and in particular no cancellations occur and these are all true poles of $f$ of order $m$. Hence $\mathrm{disp}(f,\tau)\geq 1$ also in this last case where $\mathrm{ord}(f,\tau)=m>\lambda$.\end{proof}

\begin{rem}\label{rem:lambda-exceptions} The exceptions in Theorem~\ref{thm:summable-dispersion} cannot be omitted. If $\lambda\neq 0$ then every $\Delta_\lambda(\frac{c}{p^\lambda-1})=c\in\mathbb{K}$ is $\lambda$-Mahler summable and has $\mathrm{disp}(c,\infty)=0$ whenever $c\neq 0$. If $\lambda\geq 1$ then for any $\gamma\in\mathcal{C}(\tau)$ with $\varepsilon(\tau)=:e\geq 1$ one can construct (cf.~Section~\ref{sec:torsion-residues}) $g=\sum_{k=1}^\lambda\sum_{\ell=0}^{e-1}c_{k,\ell}\cdot(x-\gamma^{p^\ell})^{-k}$ such that $\mathrm{disp}(\Delta_\lambda(g),\tau)=0$. The simplest such example is with $\lambda,\gamma,e=1$ (and $p\in\mathbb{Z}_{\geq 2}$ still arbitrary): \[f:=\Delta_1\left(\frac{1}{x-1}\right)=\frac{p}{x^p-1}-\frac{1}{x-1}=\frac{pV^1_{1,1}(1)-1}{x-1}+\sum_{i=1}^{p-1}\frac{pV^1_{1,1}(\zeta_p^i)}{x-\zeta_p^i}=\sum_{i=1}^{p-1}\frac{\zeta_p^i}{x-\zeta_p^i},\] which is $1$-Mahler summable but has $\mathrm{disp}(f,\tau(1))=0$. More generally, all other such examples for arbitrary $\lambda\geq 1$  and $\tau\in\mathcal{T}_+$, of $f\in\mathbb{K}(x)$ such that $f_\tau$ is $\lambda$-Mahler summable but $\mathrm{disp}(f,\tau)=0$, arise essentially from the basic construction $f_\tau:=\Delta_\lambda(g_\tau)$ with \[g_\tau=\sum_{k=1}^\lambda\sum_{\gamma\in\mathcal{C}(\tau)}\frac{\omega \cdot w^{(\lambda)}_k(\gamma)}{(x-\gamma)^k}\] for an arbitrary constant $0\neq \omega\in\mathbb{K}$ and the vector $\mathbf{w}^{(\lambda)}=(w^{(\lambda)}_k(\gamma))$ defined in Lemma~\ref{lem:d-map-kernel}.
\end{rem}


\section{Twisted Mahler discrete residues}\label{sec:residues}

Our goal in this section is to define the $\lambda$-Mahler discrete residues of $f(x)\in\mathbb{K}(x)$ for $\lambda\in\mathbb{Z}$ and prove our Main~Theorem in Section~\ref{sec:proof}, that these $\lambda$\nobreakdash-Mahler discrete residues comprise a complete obstruction to $\lambda$\nobreakdash-Mahler summability. We begin with the relatively simple construction of $\lambda$-Mahler discrete residues at $\infty$ (for Laurent polynomials), followed by the construction of $\lambda$-Mahler discrete residues at Mahler trees $\tau\in\mathcal{T}=\mathcal{T}_0\cup\mathcal{T}_+$ (see Definition~\ref{defn:cycles}), first for non-torsion $\tau\in\mathcal{T}_0$, and finally for torsion $\tau\in\mathcal{T}_+$, in increasing order of complexity, and prove separately in each case that these $\lambda$-Mahler discrete residues comprise a complete obstruction to the $\lambda$-Mahler summability of the corresponding components of $f$.


\subsection{Twisted Mahler discrete residues at infinity}\label{sec:laurent-residues}

We now define the $\lambda$-Mahler discrete residue of $f\in\mathbb{K}(x)$ at $\infty$ in terms of the Laurent polynomial component $f_{\infty}\in\mathbb{K}[x,x^{-1}]$ of $f$ in \eqref{eq:f-decomposition}, and show that it forms a complete obstruction to the $\lambda$-Mahler summability of $f_{\infty}$. The definition and proof in this case are both straightforward, but they provide helpful moral guidance for the analogous definitions and proofs in the case of $\lambda$-Mahler discrete residues at Mahler trees $\tau\in\mathcal{T}$.

\begin{defn}\label{defn:laurent-residues} For $f\in\mathbb{K}(x)$ and $\lambda\in\mathbb{Z}$, the $\lambda$-\emph{Mahler discrete residue} of $f$ at $\infty$ is the vector \[\mathrm{dres}_\lambda(f,\infty)=\Bigl(\mathrm{dres}_\lambda(f,\infty)_\theta\Bigr)_{\theta\,\in\,\mathbb{Z}/\mathcal{P}}\in\bigoplus_{\theta\,\in\,\mathbb{Z}/\mathcal{P}}\mathbb{K}\] defined as follows. Write $f_{\infty}=\sum_{\theta\,\in\,\mathbb{Z}/\mathcal{P}}f_\theta$ as in Definition~\ref{defn:trajectory-projection}, and write each component $f_\theta=\sum_{j= 0}^{h_\theta} c_{ip^j}x^{ip^j}$ with $p\nmid i$ whenever $i\neq 0$ (that is, with each $i$ initial in its maximal $\mathcal{P}$-trajectory $\theta$), and where $h_\theta=0$ if $f_\theta=0$ and otherwise $h_\theta\in\mathbb{Z}_{\geq 0}$ is as large as possible such that $c_{ip^{h_\theta}}\neq 0$. Then we set
\[ \mathrm{dres}_\lambda(f,\infty)_\theta:=p^{\lambda h_\theta}\sum_{j= 0}^{h_\theta} p^{-\lambda j}c_{ip^j} \quad \text{for} \ \theta\neq\{0\}; \qquad \text{and}\qquad \mathrm{dres}_\lambda(f,\infty)_{\{0\}}:=\begin{cases} c_0 & \text{if} \ \lambda=0; \\ 0 & \text{if} \ \lambda\neq 0.\end{cases}\]\end{defn}

\begin{prop}\label{prop:laurent-residues} For $f\in\mathbb{K}(x)$ and $\lambda\in\mathbb{Z}$, the component $f_{\infty}\in \mathbb{K}[x,x^{-1}]$ in \eqref{eq:f-decomposition} is $\lambda$-Mahler summable if and only if $\mathrm{dres}_\lambda(f,\infty)=\mathbf{0}$.
\end{prop}

\begin{proof} By Lemma~\ref{lem:trajectory-projection}, $f_{\infty}$ is $\lambda$-Mahler summable if and only if $f_\theta$ is $\lambda$-Mahler summable for all $\theta\in\mathbb{Z}/\mathcal{P}$. We shall show that $f_\theta$ is $\lambda$-Mahler summable if and only if \mbox{$\mathrm{dres}_\lambda(f,\infty)_\theta=0$}. If $\lambda\neq 0$ then $f_{\{0\}}=\Delta_\lambda(\frac{c_0}{p^\lambda-1})$ is always $\lambda$-Mahler summable, whilst we have defined $\mathrm{dres}_\lambda(f,\infty)_{\{0\}}=0$ in this case. On the other hand, for $\lambda=0$, $f_{\{0\}}=\mathrm{dres}_0(f,\infty)_{\{0\}}$, and $\mathrm{disp}(f_{\{0\}},\infty)=0$ if $f_{\{0\}}\neq 0$, whilst if $f_{\{0\}}=0$ then it is clearly $\lambda$-Mahler summable. By Theorem~\ref{thm:summable-dispersion} in case $\lambda=0$, and trivially in case $\lambda\neq 0$, we conclude that $f_{\{0\}}$ is $\lambda$-Mahler summable if and only if $\mathrm{dres}_\lambda(f,\infty)_{\{0\}}=0$.

Now let us assume $\{0\}\neq \theta\in\mathbb{Z}/\mathcal{P}$ and let us write $ f_\theta=\sum_{j\geq 0} c_{ip^j}x^{ip^j}\in\mathbb{K}[x,x^{-1}]_\theta$, for the unique minimal $i\in \theta$ such that $p\nmid i$. If $f_\theta=0$ then we have nothing to show, so suppose $f_\theta\neq 0$ and let $h_\theta\in\mathbb{Z}_{\geq 0}$ be maximal such that $c_{ip^{h_\theta}}\neq 0$. Letting $\Delta^{(n)}_\lambda:=p^{\lambda n}\sigma^n -\mathrm{id}$ as in Lemma~\ref{lem:mahler-coefficients-computations}, we find that \[\bar{f}_{\lambda,\theta}:=f_\theta+\sum_{j=0}^{h_\theta}\Delta_\lambda^{(h_\theta-j)}(c_{ip^j}x^{ip^j})=\sum_{j=0}^{h_\theta}p^{\lambda(h_\theta-j)}c_{ip^j}x^{ip^{h_\theta}}+0=\mathrm{dres}_\lambda(f,\infty)_\theta\cdot x^{ip^{h_\theta}}.\] By Lemma~\ref{lem:mahler-coefficients-computations}, we see that $f_\theta$ is $\lambda$-Mahler summable if and only if $\bar{f}_{\lambda,\theta}$ is $\lambda$-Mahler summable.
Clearly, $\bar{f}_{\lambda,\theta}=0$ if and only if $\mathrm{dres}(f,\infty)_\theta=0$. We also see that $\mathrm{disp}(\bar{f}_{\lambda,\theta},\infty)=0$ if $\mathrm{dres}_\lambda(f,\infty)_\theta\neq 0$, in which case $\bar{f}_{\lambda,\theta}$ cannot be $\lambda$-Mahler summable by Theorem~\ref{thm:summable-dispersion}, and so $f_\theta$ cannot be $\lambda$-Mahler summable either. On the other hand, if $\bar{f}_{\lambda,\theta}=0$ then $f_\theta$ is $\lambda$-Mahler summable by Lemma~\ref{lem:mahler-coefficients-computations}. \end{proof}

\begin{rem} \label{rem:laurent-remainder} The factor of $p^{\lambda h_\theta}$ in the Definition~\ref{defn:laurent-residues} of $\mathrm{dres}_\lambda(f,\infty)_\theta$ for $\{0\}\neq \theta\in\mathbb{Z}/\mathcal{P}$ plays no role in deciding whether $f_{\infty}$ is $\lambda$-Mahler summable, but this normalization allows us to define uniformly the $\bar{f}_{\lambda,\theta}=\mathrm{dres}_\lambda(f,\infty)_\theta\cdot x^{ip^{h_\theta}}$ as the $\theta$-component of the $\bar{f}_\lambda\in \mathbb{K}(x)$ in the $\lambda$-\emph{Mahler reduction}~\eqref{eq:mahler-reduction}. For every $\{0\}\neq\theta\in\mathbb{Z}/\mathcal{P}$, we set $h_\theta(f)$ to be the $h_\theta$ defined in the course of the proof of Proposition~\ref{prop:laurent-residues} in case $f_\theta\neq 0$, and in all other cases we set $h_\theta(f):=0$.
\end{rem}


\subsection{Twisted Mahler discrete residues at Mahler trees: the non-torsion case} \label{sec:non-torsion-residues}

We now define the $\lambda$-Mahler discrete residues of $f\in\mathbb{K}(x)$ at non-torsion Mahler trees $\tau\in\mathcal{T}_0$ 
in terms of the partial fraction decomposition of the component $f_\tau\in\mathbb{K}(x)_\tau$ in Definition~\ref{defn:tree-projection}, and show that it forms a complete obstruction to the $\lambda$-Mahler summability of $f_\tau$. 

We begin by introducing some auxiliary notion, which already appeared in \cite{arreche-zhang:2022}, but with an unfortunately different choice of notation.

\begin{defn}\label{defn:bouquet} Let $\tau\in\mathcal{T}_0$, $\gamma\in\tau$, and $h\in\mathbb{Z}_{\geq 0}$. The \emph{bouquet} of height $h$ rooted at $\gamma$ is \[\beta_h(\gamma):=\left\{\alpha\in\tau \ | \ \alpha^{p^n}=\gamma \ \text{for some} \ 0\leq n \leq h\right\}.\]
\end{defn}

\begin{lem}[cf.~\protect{\cite[Lem.~4.4]{arreche-zhang:2022}}]\label{lem:bouquet} Let $\tau\in\mathcal{T}_0$ and $S\subset \tau$ be a finite non-empty subset. Then there exists a unique $\gamma\in\tau$ such that $S\subseteq\beta_h(\gamma)$ with $h$ as small as possible.
\end{lem}
\begin{proof} This is an immediate consequence of the proof of \cite[Lem.~4.4]{arreche-zhang:2022}, whose focus and notation was rather different from the one adopted here, so let us complement it here with an alternative and more conceptual argument. As explained in \cite[Remark~2.7 and Example~2.9]{arreche-zhang:2022}, we can introduce a digraph structure on $\tau$ in which we have a directed edge $\alpha\rightarrow \xi$ whenever $\alpha^p=\xi$, resulting in an infinite (directed) tree. The ``meet'' of the elements of $S$ is the unique $\gamma\in\tau$ such that $S\subseteq\beta_h(\gamma)$ with $h$ as small as possible. 
\end{proof}

\begin{defn}[cf.~\protect{\cite[Def.~4.6]{arreche-zhang:2022}}]\label{defn:non-torsion-height} For $f\in\mathbb{K}(x)$ and $\tau\in\mathrm{supp}(f)\cap\mathcal{T}_0$, the \emph{height} of $f$ at $\tau$, denoted by $\mathrm{ht}(f,\tau)$, is the smallest $h\in\mathbb{Z}_{\geq 0}$ such that $\mathrm{sing}(f,\tau)\subseteq\beta_h(\gamma)$ for the unique $\gamma\in\tau$ identified in Lemma~\ref{lem:bouquet} with $S=\mathrm{sing}(f,\tau)\subset\tau$. We write $\beta(f,\tau):=\beta_h(\gamma)$, the \emph{bouquet} of $f$ in $\tau$. For $\alpha\in\beta(f,\tau)$, the \emph{height} of $\alpha$ in $f$, denoted by $\eta(\alpha|f)$, is the unique $0\leq n \leq h$ such that $\alpha^{p^n}=\gamma$.\end{defn}

In \cite[Def.~4.10]{arreche-zhang:2022} we gave a recursive definition in the $\lambda=0$ case of Mahler discrete residues for non-torsion $\tau\in\mathcal{T}_0$. Here we provide a non-recursive definition for $\lambda\in\mathbb{Z}$ arbitrary, which can be shown to agree with the one from \cite{arreche-zhang:2022} in the special case $\lambda=0$.

\begin{defn}\label{defn:non-torsion-residues} For $f\in\mathbb{K}(x)$, $\lambda\in\mathbb{Z}$, and $\tau\in\mathcal{T}_0$, the $\lambda$-\emph{Mahler discrete residue} of $f$ at $\tau$ of degree $k\in\mathbb{N}$ is the vector \[\mathrm{dres}_\lambda(f,\tau,k)=\Bigl(\mathrm{dres}_\lambda(f,\tau,k)_\alpha\Bigr)_{\alpha\in \tau}\in\bigoplus_{\alpha\in\tau}\mathbb{K}\] defined as follows.

We set $\mathrm{dres}_\lambda(f,\tau,k)=\mathbf{0}$ if either $\tau\notin\mathrm{supp}(f)$ or $k>\mathrm{ord}(f,\tau)$ as in Definition~\ref{defn:supp}. For $\tau\in\mathrm{supp}(f)$, let us write \begin{equation}\label{eq:non-torsion-expansion}f_\tau=\sum_{k\in\mathbb{N}} \ \sum_{\alpha\in \tau}\frac{c_k(\alpha)}{(x-\alpha)^k}.\end{equation} We set $\mathrm{dres}_\lambda(f,\tau,k)_\alpha=0$ for every $k\in \mathbb{N}$ whenever $\alpha\in\tau$ is such that either $\alpha\notin \beta(f,\tau)$ or, for $\alpha\in\beta(f,\tau)$, such that $\eta(\alpha|f)\neq h$, where $h:=\mathrm{ht}(f,\tau)$ and $\beta(f,\tau)$ are as in Definition~\ref{defn:non-torsion-height}.

Finally, for the remaining $\alpha\in\beta(f,\tau)$ with $\eta(\alpha|f)=h$ and $1\leq k\leq \mathrm{ord}(f,\tau)=:m$, we define \begin{equation}\label{eq:non-torsion-residues}\mathrm{dres}_\lambda(f,\tau,k)_\alpha:=\sum_{s=k}^m  \sum_{n=0}^hp^{\lambda n}V^s_{k,n}(\alpha)c_s(\alpha^{p^n}),\end{equation}where the Mahler coefficients $V^s_{k,n}(\alpha)$ are as in Proposition~\ref{prop:mahler-coefficients}.
\end{defn}

\begin{prop} \label{prop:non-torsion-residues} For $f\in\mathbb{K}(x)$, $\lambda\in\mathbb{Z}$, and $\tau\in\mathcal{T}_0$, the component $f_\tau$ is $\lambda$-Mahler summable if and only if $\mathrm{dres}_\lambda(f,\tau,k)=\mathbf{0}$ for every $k\in\mathbb{N}$.
\end{prop}

\begin{proof} The statement is trivial for $\tau\notin\mathrm{supp}(f)\Leftrightarrow f_\tau=0$. So let us suppose $\tau\in\mathrm{supp}(f)$, and let $h:=\mathrm{ht}(f,\tau)$, $m:=\mathrm{ord}(f,\tau)$, and $\eta(\alpha):=\eta(\alpha|f)$ for each $\alpha\in\beta(f,\tau)$. 
Writing $f_\tau$ as in \eqref{eq:non-torsion-expansion}, let us also write, for $0\leq n\leq h$,
\[ f_\tau^{(n)}:=\sum_{k=1}^m\sum_{\substack{\alpha\in\beta(f,\tau)\\\eta(\alpha)=n}}\frac{c_k(\alpha)}{(x-\alpha)^k}\qquad\text{so that}\qquad f_\tau=\sum_{n=0}^h f_\tau^{(n)}.\] 

By Lemma~\ref{lem:mahler-coefficients-computations}, for each $0\leq n\leq h$ 
we have
\[\sigma^{n}\left(f_\tau^{(h-n)}\right)
=\sum_{k=1}^m\sum_{\substack{\alpha\in\beta(f,\tau)\tau \\ \eta(\alpha)=h}}\frac{\sum_{s=k}^mV^s_{k,n}(\alpha)c_s(\alpha^{p^n})}{(x-\alpha)^k},\] and therefore \[\Delta_\lambda^{(n)}\left(f_\tau^{(h-n)}\right)
=- f_\tau^{(h-n)}+\sum_{k=1}^m\sum_{\substack{\alpha\in\beta(f,\tau) \\ \eta(\alpha)=h}}\frac{p^{\lambda n}\sum_{s=k}^mV^s_{k,n}(\alpha)c_s(\alpha^{p^n})}{(x-\alpha)^k}.\] It follows from the Definition~\ref{defn:non-torsion-residues} that \begin{equation}\label{eq:non-torsion-remainder}\bar{f}_\tau:=f_\tau+\sum_{n=0}^n\Delta_\lambda^{(n)}\left(f_\tau^{(h-n)}\right)=\sum_{k=1}^m\sum_{\alpha\in\tau}\frac{\mathrm{dres}_\lambda(f,\tau,k)_\alpha}{(x-\alpha)^k}.\end{equation}

By Lemma~\ref{lem:mahler-coefficients-computations}, $\bar{f}_{\lambda,\tau}-f_\tau$ is $\lambda$-Mahler summable, and therefore $f_\tau$ is $\lambda$-Mahler summable if and only if $\bar{f}_{\lambda,\tau}$ is $\lambda$-Mahler summable. If $\mathrm{dres}_\lambda(f,\tau,k)=\mathbf{0}$ for every $1\leq k\leq m$, then $\bar{f}_{\lambda,\tau}=0$ and therefore $f_\tau$ is $\lambda$-Mahler summable. On the other hand, if some $\mathrm{dres}_\lambda(f,\tau,k)\neq \mathbf{0}$, then $0\neq \bar{f}_{\lambda,\tau}$ has $\mathrm{disp}(\bar{f}_{\lambda,\tau},\tau)=0$ (see Definition~\ref{defn:dispersion}), whence by Theorem~\ref{thm:summable-dispersion} $\bar{f}_{\lambda,\tau}$ could not possibly be $\lambda$-Mahler summable, and therefore neither could $f_\tau$. This concludes the proof that $f_\tau$ is $\lambda$-Mahler summable if and only if $\mathrm{dres}_\lambda(f,\tau,k)=\mathbf{0}$ for every $l\in\mathbb{N}$. \end{proof}

\begin{rem}\label{rem:non-torsion-remainder} For $f\in\mathbb{K}(x)$ and $\tau\in\mathrm{supp}(f)\cap\mathcal{T}_0$, the element $\bar{f}_{\lambda,\tau}$ in \eqref{eq:non-torsion-remainder} is the $\tau$-component of the $\bar{f}_\lambda\in\mathbb{K}(x)$ in the $\lambda$-\emph{Mahler reduction} \eqref{eq:mahler-reduction}. 
\end{rem}


\subsection{Twisted Mahler discrete residues at Mahler trees: the torsion case} \label{sec:torsion-residues}

We now define the $\lambda$-Mahler discrete residues of $f\in\mathbb{K}(x)$ at torsion trees $\tau\in\mathcal{T}_+$ (see Definition~\ref{defn:cycles}) in terms of the partial fraction decomposition of the component $f_\tau\in\mathbb{K}(x)_\tau$ in Definition~\ref{defn:tree-projection}, and show that it forms a complete obstruction to the $\lambda$-Mahler summability of $f_\tau$. The definitions and proofs in this case are more technical than in the non-torsion case, involving the cycle map $\mathcal{D}_{\lambda,\tau}$ of Definition~\ref{defn:d-map} and its $\omega$-section $\mathcal{I}^{(\omega)}_{\lambda,\tau}$ from Definition~\ref{defn:omega-section}, for a particular choice of constant $\omega\in\mathbb{K}$ associated to $f$, which we construct in Definition~\ref{defn:omega-average}.

We begin by recalling the following definition from \cite{arreche-zhang:2022}, which is the torsion analogue of Definition~\ref{defn:non-torsion-height}.

\begin{defn}[cf.~\protect{\cite[Def.~4.6]{arreche-zhang:2022}}] \label{defn:torsion-height} For $\tau\in\mathcal{T}_+$ and $\alpha\in\tau$, the \emph{height} of $\alpha$, denoted by $\eta(\alpha)$, is the smallest $n\in\mathbb{Z}_{\geq 0}$ such that $\alpha^{p^n}\in\mathcal{C}(\tau)$ (cf.~Definition~\ref{defn:cycles}). For $f\in\mathbb{K}(x)$ and $\tau\in\mathrm{supp}(f)\cap\mathcal{T}_+$, the \emph{height} of $f$ at $\tau$ is \[\mathrm{ht}(f,\tau):=\mathrm{max}\{\eta(\alpha) \ | \ \alpha\in\mathrm{sing}(f,\tau)\},\] or equivalently, the smallest $h\in\mathbb{Z}_{\geq 0}$ such that $\alpha^{p^h}\in\mathcal{C}(\tau)$ for every pole $\alpha$ of $f$ in $\tau$.
\end{defn}

The following definition will allow us to use the correct $\omega$-section $\mathcal{I}^{(\omega)}_{\lambda,\tau}$ from Definition~\ref{defn:omega-section} in our construction of $\lambda$-Mahler discrete residues in the torsion case.

\begin{defn} \label{defn:omega-average} For $f\in\mathbb{K}(x)$ and $\tau\in\mathrm{supp}(f)\cap\mathcal{T}_+$, let us write \[f_\tau=\sum_{k\in\mathbb{N}}\sum_{\alpha\in\tau}\frac{c_k(\alpha)}{(x-\alpha)^k}.\] For $\lambda\in\mathbb{Z}$, we define the \emph{residual average} $\omega_{\lambda,\tau}(f)\in\mathbb{K}$ of $f$ (relative to $\lambda$ and $\tau$) as follows.

If $\lambda\leq 0$ or if $h:=\mathrm{ht}(f,\tau)=0$ (cf.~Definition~\ref{defn:torsion-height}), we simply set $\omega_{\lambda,\tau}(f)=0$. In case both $\lambda,h\geq 1$, let $\tau_h:=\{\alpha\in\tau \ | \ \eta(\alpha)=h\}$ be the set of elements of $\tau$ of height $h$. Let us write $\mathbf{c}=(c_k(\gamma))$, for $\gamma$ ranging over $\mathcal{C}(\tau)$ only, and let $(d_k^{(0)}(\gamma))=\mathbf{d}^{(0)}:=\mathcal{I}^{(0)}_{\lambda,\tau}(\mathbf{c})$ as in Definition~\ref{defn:omega-section} and $(\tilde{c}_k(\gamma))=\tilde{\mathbf{c}}=\mathcal{D}_{\lambda,\tau}(\mathbf{d}^{(0)})$, as in Definition~\ref{defn:omega-section}. Then we define
\begin{multline}\label{eq:omega-average}
\omega_{\lambda,\tau}(f):=\frac{1}{(p^h-p^{h-1})e}\sum_{\alpha\in\tau_h}\sum_{s\geq\lambda}\sum_{n=0}^{h-1}p^{\lambda n}\mathbb{V}^s_{\lambda,n}\alpha^{-sp^n}c_s(\alpha^{p^n})\\
-\frac{p^{\lambda(h-1)}}{e}\sum_{\gamma\in\mathcal{C}(\tau)}\sum_{s\geq \lambda}\mathbb{V}^s_{\lambda,h-1}\gamma^{-s}(\tilde{c}_s(\gamma)+d_s^{(0)}(\gamma)),\end{multline}

where the universal Mahler coefficients $\mathbb{V}^s_{\lambda,n}\in\mathbb{Q}$ are defined as in Section~\ref{sec:mahler-coefficients}. 
\end{defn}

The significance of this definition of the residual average $\omega_{\lambda,\tau}(f)$ and our choice of nomenclature is explained in the proof of Proposition~\ref{prop:torsion-residues} below (with the aid of Lemma~\ref{lem:uniform-height}). We are now ready to define the $\lambda$-Mahler discrete residues at torsion Mahler trees. In \cite[Def.~4.16]{arreche-zhang:2022} we gave a \emph{recursive} definition of Mahler discrete residues for torsion $\tau\in\mathcal{T}_+$ in the $\lambda=0$ case. Here we provide a less recursive definition for $\lambda\in\mathbb{Z}$ arbitrary, which can be shown to agree with the one from \cite{arreche-zhang:2022} in the special case $\lambda=0$. This new definition is only \emph{less} recursive than that of \cite{arreche-zhang:2022} because of the intervention of the map $\mathcal{I}_{\lambda,\tau}^{(\omega)}$, for which we have not found a closed form and whose definition is still essentially recursive.

\begin{defn}\label{defn:torsion-residues}For $f\in\mathbb{K}(x)$, $\lambda\in\mathbb{Z}$, and $\tau\in\mathcal{T}$ with $\tau\subset\mathbb{K}^\times_t$, the $\lambda$-\emph{Mahler discrete residue} of $f$ at $\tau$ of degree $k\in\mathbb{N}$ is the vector \[\mathrm{dres}_\lambda(f,\tau,k)=\Bigl(\mathrm{dres}_\lambda(f,\tau,k)_\alpha\Bigr)_{\alpha\in\tau}\in\bigoplus_{\alpha\in\tau}\mathbb{K}\] defined as follows.

We set $\mathrm{dres}_\lambda(f,\tau,k)=\mathbf{0}$ if either $\tau\notin\mathrm{supp}(f)$ or $k>\mathrm{ord}(f,\tau)$ as in Definition~\ref{defn:supp}. For $\tau\in\mathrm{supp}(f)$, let us write \begin{equation}\label{eq:torsion-expansion} f_\tau=\sum_{k\in\mathbb{N}}\sum_{\alpha\in\tau}\frac{c_k(\alpha)}{(x-\alpha)^k}.\end{equation} We set $\mathrm{dres}_\lambda(f,\tau,k)_\alpha=0$ for every $k\in\mathbb{N}-\{\lambda\}$ whenever $\alpha\in\tau$ is such that $\eta(\alpha)\neq h$, where $h:=\mathrm{ht}(f,\tau)$ and $\eta(\alpha)$ are as in Definition~\ref{defn:torsion-height}. In case $\lambda\geq 1$, we set $\mathrm{dres}_\lambda(f,\tau,\lambda)_\alpha=0$ also whenever $\eta(\alpha)\notin\{0,h\}$.

In case $h=0$, so that $\mathrm{sing}(f,\tau)\subseteq\mathcal{C}(\tau)$, we simply set \[\mathrm{dres}_\lambda(f,\tau,k)_\gamma:=c_k(\gamma)\] for every $1\leq k\leq\mathrm{ord}(f,\tau)$ and $\gamma\in\mathcal{C}(\tau)$. In case $h\geq 1$, let us write $\mathbf{c}=(c_k(\gamma))$ for $\gamma$ ranging over $\mathcal{C}(\tau)$ only, and let $(d_k(\gamma))=\mathbf{d}:=\mathcal{I}_{\lambda,\tau}^{(\omega)}(\mathbf{c})$ as in Definition~\ref{defn:omega-section}, where $\omega:=\omega_{\lambda,\tau}(f)$ (cf.~Definition~\ref{defn:omega-average}), and $(\tilde{c}_k(\gamma))=\tilde{\mathbf{c}}:=\mathcal{D}_{\lambda,\tau}(\mathbf{d})$ as in Definition~\ref{defn:d-map}. For $\alpha\in\tau$ such that $\eta(\alpha)=h$ and for $1\leq k\leq\mathrm{ord}(f,\tau)=:m$, we define \begin{multline}\label{eq:torsion-residues}\mathrm{dres}_\lambda(f,\tau,k)_\alpha:=\sum_{s=k}^m\sum_{n=0}^{h-1}p^{\lambda n}V^s_{k,n}(\alpha)c_s(\alpha^{p^n}) \\ -p^{\lambda(h-1)}\sum_{s=k}^m\mathbb{V}^s_{k,h-1}\alpha^{k-sp^{h+e-1}}\left(\tilde{c}_s\left(\alpha^{p^{h+e-1}}\right)+d_s\left(\alpha^{p^{h+e-1}}\right)
\right).\end{multline} In case $\lambda\geq 1$, for $\gamma\in\mathcal{C}(\tau)$ we set \begin{equation}\label{eq:torsion-residues-cycle}\mathrm{dres}_\lambda(f,\tau,\lambda)_\gamma:=c_\lambda(\gamma)-\tilde{c}_\lambda(\gamma)=\frac{\gamma^\lambda}{e}\sum_{j=1}^e \gamma^{-\lambda p^j}\left(c_\lambda(\gamma^{p^j})-p^\lambda\sum_{s\geq \lambda+1} V^s_{\lambda,1}(\gamma^{p^j})d_s(\gamma^{p^{j+1}})\right).\end{equation}
\end{defn}

\begin{rem}\label{rem:torsion-residues-computation} The Definition~\ref{defn:torsion-residues} can be expressed equivalently in ways that are easier to compute, but which require a lot of hedging. We cannot improve on the definition in case $h=0$; so let us address the case $h\geq 1$. The different ingredients used in Definition~\ref{defn:torsion-residues} are best computed in the following order. In every case, one should first compute the vector $\mathbf{d}^{(0)}:=\mathcal{I}^{(0)}_{\lambda,\tau}(\mathbf{c})$ of Proposition~\ref{prop:omega-section}. Every instance of $\tilde{c}_s$ in \eqref{eq:omega-average} and in \eqref{eq:torsion-residues} can (and should) be replaced with $c_s$, with the single exception of $\tilde{c}_\lambda$ (if it happens to occur), which should be rewritten in terms of the $c_s$ and $d_s^{(0)}$ using \eqref{eq:omega-section-lambda}. There is no need to find $\tilde{\mathbf{c}}$ by applying $\mathcal{D}_{\lambda,\tau}$ to anything. Having made these replacements, and only then, one should then compute the residual average $\omega$ from Definition~\ref{defn:omega-average}. If this $\omega$ happens to be $0$ then we already have all the required ingredients to compute our discrete residues. Only in case $\omega\neq 0$, we then proceed to compute the vector $\mathbf{w}^{(\lambda)}$ of Lemma~\ref{lem:d-map-kernel}, and by Definition~\ref{defn:omega-section} we can replace the $d_s$ in \eqref{eq:torsion-residues} with $d_s^{(0)}+\omega\cdot w^{(\lambda)}_s$, all of which have already been computed, and now we are once again in possession of all the required ingredients. \end{rem}

We next present several preparatory Lemmas that will aid us in streamlining our proof of Proposition~\ref{prop:torsion-residues} below that the $\lambda$-Mahler discrete residues just defined comprise a complete obstruction to the $\lambda$-Mahler summability of $f_\tau$ for $\tau\in\mathcal{T}_+$. We hope that the reader who, like us, finds the above Definition~\ref{defn:torsion-residues} painfully complicated, especially in comparison with the relatively simpler Definition~\ref{defn:non-torsion-residues} in the non-torsion case, can begin to glimpse in the statements of the following preliminary results the reasons for the emergence of the additional ingredients in Definition~\ref{defn:torsion-residues} that are absent from Definition~\ref{defn:non-torsion-residues}. This is why we have chosen to present them first, and postpone their proofs until after their usefulness has become apparent in the proof of Proposition~\ref{prop:torsion-residues}.

\begin{lem}\label{lem:height-0} Let $f\in\mathbb{K}(x)$, $\lambda\in\mathbb{Z}$, and $\tau\in\mathrm{supp}(f,\tau)\cap\mathcal{T}_+$. If $\mathrm{ht}(f,\tau)=0$ then $f_\tau$ is not $\lambda$-Mahler summable.
\end{lem}

\begin{lem}\label{lem:g0g1} Let $\lambda\in\mathbb{Z}$ and $\tau\in\mathcal{T}_+$, and set $e:=|\mathcal{C}(\tau)|$ as in~Definition~\ref{defn:cycles}. Let $f\in\mathbb{K}(x)$, and let us write the cyclic component \[\mathcal{C}(f_\tau)=\sum_{k\in\mathbb{N}}\sum_{\gamma\in\mathcal{C}(\tau)}\frac{c_k(\gamma)}{(x-\gamma)^k},\] as in Definition~\ref{defn:cyclic-component}, and let us write $\mathbf{c}=(c_k(\gamma))\in\mathcal{S}^{\mathcal{C}(\tau)}$. Let $\omega\in\mathbb{K}$ be arbitrary, and let us write $\mathbf{d}=(d_k(\gamma))=\mathcal{I}_{\lambda,\tau}^{(\omega)}(\mathbf{c})$ as in Definition~\ref{defn:omega-section} and $\tilde{\mathbf{c}}=\mathcal{D}_{\lambda,\tau}(\mathbf{d})$ as in Definition~\ref{defn:d-map}. Set \begin{equation}\label{eq:g0g1-defn}g_0:=\sum_{k\in\mathbb{N}}\sum_{\gamma\in\mathcal{C}(\tau)}\frac{d_k(\gamma)}{(x-\gamma)^k}\qquad\text{and}\qquad g_1:=-\sum_{k\in\mathbb{N}}\sum_{\gamma\in\mathcal{C}(\tau)}\sum_{i=1}^{p-1}\frac{\zeta_p^{ki}(\tilde{c}_k(\gamma)+d_k(\gamma))}{(x-\zeta_p^i\gamma)^k}.\end{equation}

Then \begin{equation}\label{eq:g0g1-cases}\mathcal{C}(f_\tau)-\Delta_\lambda(g_0)=\begin{cases}
    g_1 & \text{if} \ \lambda\leq 0;\\
    \displaystyle g_1+\sum_{\gamma\in\mathcal{C}(\tau)}\frac{c_\lambda(\gamma)-\tilde{c}(\gamma)}{(x-\gamma)^\lambda}& \text{if} \ \lambda\geq 1.\end{cases}
\end{equation} Moreover, for any $h\geq 1$, writing $\tau_h:=\{\alpha\in\tau \ | \ \eta(\alpha)=h\}$, we have 

\begin{equation}\label{eq:g1-h}\sigma^{h-1}(g_1)=-\sum_{k\in\mathbb{N}}\sum_{\alpha\in\tau_h}\frac{\sum_{s\geq k}\mathbb{V}^s_{k,h-1}\alpha^{k-sp^{h+e-1}}\left(\tilde{c}_s\left(\alpha^{p^{h+e-1}}\right)+d_s\left(\alpha^{p^{h+e-1}}\right)\right)}{(x-\alpha)^k}.\end{equation}
\end{lem}

\begin{lem}\label{lem:uniform-height} Let $\lambda\geq 1$, $h\geq 1$, $\bar{f}_\tau\in\mathbb{K}(x)_\tau$, and $\tau\in\mathrm{supp}(\bar{f})\cap\mathcal{T}_+$ such that $\mathrm{ord}(\bar{f},\tau)=\lambda$ and $\mathrm{sing}(f,\tau)\subseteq\tau_h=\{\alpha\in\tau \ | \ \eta(\alpha)= h\}$, so that we can write \[\bar{f}_\tau=\sum_{k=1}^\lambda\sum_{\alpha\in\tau_h}\frac{\bar{c}_k(\alpha)}{(x-\alpha)^k}.\] If $\bar{f}_\tau$ is $\lambda$-Mahler summable then all the elements $\alpha^{-\lambda}\bar{c}_\lambda(\alpha)$ are equal to the constant \[\bar{\omega}=\frac{1}{|\tau_h|}\sum_{\alpha\in\tau_h}\alpha^{-\lambda}\bar{c}_\lambda(\alpha),\] which is their arithmetic average. Letting $e:=|\mathcal{C}(\tau)|$, we have $|\tau_h|=(p^h-p^{h-1})e$.
\end{lem}

\begin{prop}\label{prop:torsion-residues} For $f\in\mathbb{K}(x)$, $\lambda\in\mathbb{Z}$, and $\tau\in\mathcal{T}_+$, the component $f_\tau$ is $\lambda$-Mahler summable if and only if $\mathrm{dres}_\lambda(f,\tau,k)=\mathbf{0}$ for every $k\in\mathbb{N}$.
\end{prop}

\begin{proof} The statement is trivial for $\tau\notin\mathrm{supp}(f)\Leftrightarrow f_\tau=0$. If $\mathrm{ht}(f,\tau)=0$ then $0\neq f_\tau$ cannot be $\lambda$-Mahler summable by Lemma~\ref{lem:height-0}, whereas in this case we defined $\mathrm{dres}(f,\tau,k)_\gamma=c_k(\gamma)$ in Definition~\ref{defn:torsion-residues}, and we obtain our conclusion vacuously in this case.

From now on we assume $\tau\in\mathrm{supp}(f)$, and let $h:=\mathrm{ht}(f,\tau)\geq 1$, $m:=\mathrm{ord}(f,\tau)$, and $\omega:=\omega_{\lambda,\tau}(f)$. Writing $f_\tau$ as in \eqref{eq:torsion-expansion}, let $\tau_n:=\{\alpha\in\tau \ | \ \eta(\alpha)=n\}$ for $n\in\mathbb{Z}_{\geq 0}$ and let us also write \[f_\tau^{(n)}:=\sum_{k=1}^m\sum_{\alpha\in\tau_n}\frac{c_k(\alpha)}{(x-\alpha)^k}\qquad\text{so that} \qquad f_\tau=\sum_{n=0}^hf_\tau^{(n)}.\] The same computation as in the proof of Proposition~\ref{prop:non-torsion-residues} yields \begin{equation}\label{eq:easy-reduction} \tilde{f}_{\lambda,\tau}:=f_\tau+\sum_{n=0}^{h-1}  \Delta_\lambda^{(n)}(f_\tau^{(h-n)})=\sum_{k=1}^m\sum_{\alpha\in\tau_h}\frac{\sum_{s\geq k}\sum_{n=0}^{h-1}p^{\lambda n}V^s_{k,n}(\alpha)c_s(\alpha^{p^n})}{(x-\alpha)^k} +\sum_{k=1}^m\sum_{\gamma\in\mathcal{C}(\tau)}\frac{c_k(\gamma)}{(x-\gamma)^k}.
\end{equation}  Let us now write, as in Definition~\ref{defn:torsion-residues}, $\mathbf{c}=(c_k(\gamma))$ for $\gamma$ ranging over $\mathcal{C}(\tau)=\tau_0$ only, $(d_k(\gamma))=\mathbf{d}:=\mathcal{I}_{\lambda,\tau}^{(\omega)}(\mathbf{c})$, and $(\tilde{c}_k(\gamma))=\tilde{\mathbf{c}}:=\mathcal{D}_{\lambda,\tau}(\mathbf{d})$.

Writing $g_0$ and $g_1$ as in \eqref{eq:g0g1-defn}, it follows from Lemma~\ref{lem:g0g1} and Definition~\ref{defn:torsion-residues} that  

 \begin{equation}\label{eq:torsion-remainder}\bar{f}_{\lambda,\tau}:=\tilde{f}_{\lambda,\tau}-\Delta_\lambda(g_0)+\Delta_\lambda^{(h-1)}(g_1)=\sum_{k=1}^m\sum_{\alpha\in\tau}\frac{\mathrm{dres}_\lambda(f,\tau,k)_\alpha}{(x-\alpha)^k}.\end{equation}

By a twofold application of Lemma~\ref{lem:mahler-coefficients-computations}, to \eqref{eq:easy-reduction} and to \eqref{eq:torsion-remainder}, we find that  \[f_\tau \ \text{is} \ \lambda\text{-Mahler summable}\Longleftrightarrow \tilde{f}_{\lambda,\tau} \  \text{is} \  \lambda\text{-Mahler summable}\Longleftrightarrow \bar{f}_{\lambda,\tau} \  \text{is} \ \lambda\text{-Mahler summable}.\] On the other hand, we see from \eqref{eq:torsion-remainder} that $\bar{f}_{\lambda,\tau}=0$ if and only if $\mathrm{dres}_\lambda(f,\tau,k)=\mathbf{0}$ for every $k\in\mathbb{N}$. Therefore we immediately conclude that if $\mathrm{dres}_\lambda(f,\tau,k)=\mathbf{0}$ for every $k\in\mathbb{N}$ then $f_\tau$ is $\lambda$-Mahler summable. Moreover, in case $\lambda\leq 0$, if $f_\tau$ is $\lambda$-Mahler summable, so that $\bar{f}_{\lambda,\tau}$ is also $\lambda$-Mahler summable, then we must have $\bar{f}_{\lambda,\tau}=0$, for otherwise we would have $\mathrm{disp}(\bar{f}_{\lambda,\tau},\tau)=0$, contradicting Theorem~\ref{thm:summable-dispersion}(2). This concludes the proof of the Propostion in case $\lambda\leq 0$.

It remains to prove the converse in the case where $\lambda\geq 1$: assuming $f_\tau$ is $\lambda$-Mahler summable, we must have $\mathrm{dres}_\lambda(f,\tau,k)=\mathbf{0}$ for every $k\in\mathbb{N}$. By Proposition~\ref{prop:omega-section}, we must have $\mathbf{c}=\tilde{\mathbf{c}}$, and therefore $\mathrm{dres}_\lambda(f,\tau,k)_\gamma=c_\lambda(\gamma)-\tilde{c}_\lambda(\gamma)=0$ for every $\gamma\in\mathcal{C}(\tau)$, whence $\mathrm{sing}(\bar{f}_{\lambda,\tau},\tau)\subseteq\tau_h$ by the Definition~\ref{defn:torsion-residues} of $\mathrm{dres}_\lambda(f,\tau,k)$. Moreover, if we had $\bar{f}_{\lambda,\tau}\neq 0$, contrary to our contention, then we would have $\mathrm{disp}(\bar{f}_{\lambda,\tau},\tau)=0$, and by Theorem~\ref{thm:summable-dispersion}(3) this can only happen in case $\mathrm{ord}(\bar{f}_{\lambda,\tau},\tau)=\lambda$. So we already conclude that $\mathrm{dres}_\lambda(f,\tau,k)=\mathbf{0}$ for every $k>\lambda$ if $f_\tau$ is $\lambda$-Mahler summable. If we can further show that $\mathrm{dres}_\lambda(f,\tau,\lambda)=\mathbf{0}$ also, then this will force $\mathrm{ord}(\bar{f}_{\lambda,\tau},\tau)\neq \lambda$ and we will be able to conclude that actually $\mathrm{dres}_\lambda(f,\tau,k)=\mathbf{0}$ for every $k\in\mathbb{N}$, as we contend, by another application of Theorem~\ref{thm:summable-dispersion}.

Thus it remains to show that if $f_\tau$ is $\lambda$-Mahler summable then $\mathrm{dres}_\lambda(f,\tau,\lambda)=\mathbf{0}$, which task will occupy us for the rest of the proof. We already know that $\mathrm{dres}_\lambda(f,\tau,k)=\mathbf{0}$ for every $k>\lambda$ and $\mathrm{dres}_\lambda(f,\tau,\lambda)_\gamma=0$ for every $\gamma\in\mathcal{C}(\tau)$, and therefore $\bar{f}_{\lambda,\tau}$ satisfies the hypotheses of Lemma~\ref{lem:uniform-height} by \eqref{eq:torsion-remainder} and the Definition~\ref{defn:torsion-residues}. So let us write $\bar{c}_k(\alpha):=\mathrm{dres}_\lambda(f,\tau,k)_\alpha$ as in Lemma~\ref{lem:uniform-height}, so that \[\bar{f}_{\lambda,\tau}=\sum_{k=1}^\lambda\sum_{\alpha\in\tau_h}\frac{\bar{c}_k(\alpha)}{(x-\alpha)^k},\] and compute the arithmetic average $\bar{\omega}$ of the elements $\alpha^{-\lambda}\bar{c}_\lambda(\alpha)$ for $\alpha$ ranging over $\tau_h$, which must be equal to $\alpha^{-\lambda}\bar{c}_\lambda(\alpha)$ for each $\alpha\in\tau_h$ by Lemma~\ref{lem:uniform-height}. Firstly, we see that \[\frac{1}{|\tau_h|}\sum_{\alpha\in\tau_h}\alpha^{-\lambda}\left(\sum_{s\geq\lambda}\sum_{n=0}^{h-1}p^{\lambda n}V^s_{\lambda,n}(\alpha)c_s(\alpha^{p^n})\right)=\frac{1}{(p^h-p^{h-1})e}\sum_{\alpha\in\tau_h}\sum_{s\geq\lambda}\sum_{n=0}^{h-1}p^{\lambda n}\mathbb{V}^s_{\lambda,n}\alpha^{-sp^n}c_s(\alpha^{p^n}),\] since $V^s_{\lambda,n}(\alpha)=\mathbb{V}^s_{\lambda,n}\cdot\alpha^{\lambda-sp^n}$ by Lemma~\ref{lem:mahler-coefficients}. Secondly, we find that in the remaining portion of the average of $\alpha^{-\lambda}\bar{c}_\lambda(\alpha)=\alpha^{-\lambda}\mathrm{dres}_\lambda(f,\tau,\lambda)_\alpha$ for $\alpha$ ranging over $\tau_h$, \begin{multline}\label{eq:second-average}\frac{1}{|\tau_h|}\sum_{\alpha\in\tau_h}\alpha^{-\lambda}\left(-p^{\lambda(h-1)}\sum_{s\geq \lambda}\mathbb{V}^s_{\lambda,h-1}\alpha^{\lambda-sp^{h+e-1}}\left(\tilde{c}_s\left(\alpha^{p^{h+e-1}}\right)+d_s\left(\alpha^{p^{h+e-1}}\right)\right)\right) \\ =\frac{-p^{\lambda(h-1)}}{(p^h-p^{h-1})e}\sum_{\alpha\in\tau_h}\sum_{s\geq\lambda}\mathbb{V}^s_{\lambda,h-1}\left(\left(\alpha^{p^h}\right)^{p^{e-1}}\right)^{-s}\left(\tilde{c}_s\left(\left(\alpha^{p^h}\right)^{p^{e-1}}\right)+d_s\left(\left(\alpha^{p^h}\right)^{p^{e-1}}\right)\right),\end{multline} the summands depend only on $\alpha^{p^h}=\gamma\in\mathcal{C}(\tau)$. For each $\gamma\in\mathcal{C}(\tau)$, the set $\{\alpha\in\tau_h \ | \ \alpha^{p^h}=\gamma\}$ has $p^h-p^{h-1}$ elements: there are $(p-1)$ distinct $p^\text{th}$-roots of $\gamma$ that do not belong to $\mathcal{C}(\tau)$, and then there are $p^{h-1}$ distinct $(p^{h-1})^\text{th}$ roots of each of those elements. Therefore the expression in \eqref{eq:second-average} is equal to the simpler \[-\frac{p^{\lambda(h-1)}}{e}\sum_{\gamma\in\mathcal{C}(\tau)}\sum_{s\geq\lambda}\mathbb{V}^s_{\lambda,h-1}\gamma^{-s}(\tilde{c}_s(\gamma)+d_s(\gamma)),\] whence the average \begin{multline}\label{eq:actual-average}\bar{\omega}:=\frac{1}{|\tau_h|}\sum_{\alpha\in\tau_h}\alpha^{-\lambda}\bar{c}_\lambda(\alpha)=\frac{1}{(p^h-p^{h-1})e}\sum_{\alpha\in\tau_h}\sum_{n=0}^{h-1}\sum_{s\geq\lambda}p^{\lambda n}\mathbb{V}^s_{\lambda,n}\alpha^{-sp^n}c_s(\alpha^{p^n})\\-\frac{p^{\lambda(h-1)}}{e}\sum_{\gamma\in\mathcal{C}(\tau)}\sum_{s\geq\lambda}\mathbb{V}^s_{\lambda,h-1}\gamma^{-s}(\tilde{c}_s(\gamma)+d_s(\gamma))).\end{multline} Note that this is not necessarily the same as the similar expression for the residual average $\omega_{\lambda,\tau}(f)$ from Definition~\ref{defn:omega-average}, which was defined with respect to $(d_k^{(0)}(\gamma))=\mathbf{d}^{(0)}:=\mathcal{I}_{\lambda,\tau}^{(0)}(\mathbf{c})$ as
\begin{multline*}
\omega_{\lambda,\tau}(f)=\frac{1}{(p^h-p^{h-1})e}\sum_{\alpha\in\tau_h}\sum_{s\geq\lambda}\sum_{n=0}^{h-1}p^{\lambda n}\mathbb{V}^s_{\lambda,n}\alpha^{-sp^n}c_s(\alpha^{p^n})\\
-\frac{p^{\lambda(h-1)}}{e}\sum_{\gamma\in\mathcal{C}(\tau)}\sum_{s\geq \lambda}\mathbb{V}^s_{\lambda,h-1}\gamma^{-s}(\tilde{c}_s(\gamma)+d_s^{(0)}(\gamma)).
\end{multline*}
And yet, 
$d_s(\gamma)=d_s^{(0)}(\gamma)$ for every $s>\lambda$ and $\gamma\in\mathcal{C}(\tau)$ by Proposition~\ref{prop:omega-section} and \[d_\lambda(\gamma)=\omega_{\lambda,\tau}(f)\cdot \gamma^{\lambda}+d_\lambda^{(0)}(\gamma)\] for each $\gamma\in\mathcal{C}(\tau)$ by the Definition~\ref{defn:omega-section} of $\mathcal{I}^{(0)}_{\lambda,\tau}$ and of $\mathcal{I}_{\lambda,\tau}^{(\omega)}$ with $\omega=\omega_{\lambda,\tau}(f)$. By Corollary~\ref{cor:mahler-coefficients}, $\mathbb{V}^\lambda_{\lambda,h-1}=p^{-\lambda(h-1)}$, and therefore we find from \eqref{eq:actual-average} that \begin{multline*}\bar{\omega}=\frac{1}{(p^h-p^{h-1})e}\sum_{\alpha\in\tau_h}\sum_{n=0}^{h-1}\sum_{s\geq\lambda}p^{\lambda n}\mathbb{V}^s_{\lambda,n}\alpha^{-sp^n}c_s(\alpha^{p^n})\\-\frac{p^{\lambda(h-1)}}{e}\sum_{\gamma\in\mathcal{C}(\tau)}\sum_{s\geq\lambda+1}\mathbb{V}^s_{\lambda,h-1}\gamma^{-s}(\tilde{c}_s(\gamma)+d^{(0)}_s(\gamma)))-\frac{p^{\lambda(h-1)}}{e}\sum_{\gamma\in\mathcal{C}(\tau)}\mathbb{V}^\lambda_{\lambda,h-1}\gamma^{-\lambda}(\tilde{c}_\lambda(\gamma)+\omega\gamma^\lambda+d_\lambda^{(0)}(\gamma))\\ = \omega-\frac{p^{\lambda(h-1)}}{e}\sum_{\gamma\in\mathcal{C}(\tau)}p^{-\lambda(h-1)}\gamma^{-\lambda}\gamma^{\lambda}\omega=\omega-\omega=0.\end{multline*} Since we must have $\bar{c}_\lambda(\alpha)=\mathrm{dres}_\lambda(f,\tau,\lambda)_\alpha=\alpha^\lambda\bar{\omega}$ in \eqref{eq:actual-average} for each $\alpha\in\tau_h$ by Lemma~\ref{lem:uniform-height}, it follows that $\mathrm{dres}_\lambda(f,\tau,\lambda)=\mathbf{0}$, concluding the proof of Proposition~\ref{prop:torsion-residues}.\end{proof}

\begin{rem}\label{rem:torsion-remainder}For $f\in\mathbb{K}(x)$ and $\tau\in\mathrm{supp}(f)\cap\mathcal{T}_+$, the element $\bar{f}_{\lambda,\tau}$ in \eqref{eq:torsion-remainder} is the $\tau$-component of the $\bar{f}_\lambda\in\mathbb{K}(x)$ in the $\lambda$-\emph{Mahler reduction} \eqref{eq:mahler-reduction}. 
\end{rem}

We conclude this section by providing the proofs of the preliminary Lemmas that we used in the proof of Proposition~\ref{prop:torsion-residues}.

\begin{proof}[Proof of Lemma~\protect{\ref{lem:height-0}}] It suffices to show that for any $g\in\mathbb{K}(x)$ such that $g_\tau\neq 0$ we have $\mathrm{ht}(\Delta_\lambda(g),\tau)\geq 1$. So let us write $m:=\mathrm{ord}(g,\tau)$, $h:=\mathrm{ht}(g,\tau)$, $\tau_n:=\{\alpha\in\tau \ | \ \eta(\alpha)=n\}$ for $n\in\mathbb{Z}_{\geq 0}$, and \[0\neq g_\tau=\sum_{k=1}^m\sum_{n=0}^h\sum_{\alpha\in\tau_n}\frac{d_k(\alpha)}{(x-\alpha)^k}.\] Then \[\Delta_\lambda(g)=\sum_{\alpha\in\tau_{h+1}}\frac{p^\lambda V^m_{m,1}(\alpha)d_m(\alpha^p)}{(x-\alpha)^m}+(\text{lower-order or lower-height terms}),\] and since $p^\lambda V^m_{m,1}(\alpha)=p^{\lambda-m}\alpha^{m-pm}$ by Corollary~\ref{cor:mahler-coefficients} and at least one $d_m(\alpha^p)\neq 0$ for some $\alpha\in\tau_{h+1}$ by assumption, we conclude that $\Delta_\lambda(g)$ has at least one pole in $\tau_{h+1}$ and therefore $\mathrm{ht}(\Delta_\lambda(g),\tau)=h+1\geq 1$, as claimed.
\end{proof}

\begin{proof}[Proof of Lemma~\protect{\ref{lem:g0g1}}] It follows from \eqref{eq:mahler-coefficients} and Lemma~\ref{lem:cyclic-component-short} that \[\Delta_\lambda(g_0)=\sum_{k\in\mathbb{N}}\sum_{\gamma\in\mathcal{C}(\tau)}\frac{\tilde{c}_k(\gamma)}{(x-\gamma)^k}+\sum_{k\in\mathbb{N}}\sum_{\gamma\in\mathcal{C}(\tau)}\sum_{i=1}^{p-1}\frac{p^\lambda\sum_{s\geq k}V^s_{k,1}(\zeta_p^i\gamma)d_s(\gamma^p)}{(x-\zeta_p^i\gamma)^k}.\] To see that \[p^\lambda\sum_{s\geq k}V^s_k(\zeta_p^i\gamma)d_s(\gamma^p)=\zeta_p^{ki}(\tilde{c}_k(\gamma)+d_k(\gamma)),\] note that by Lemma~\ref{lem:mahler-coefficients} \[V^s_{k,1}(\zeta_p^i\gamma)=(\zeta_p^i\gamma)^{k-sp}\cdot\mathbb{V}^{s}_{k,1}=\zeta_p^{ki}V^s_{k,1}(\gamma)\] for every $s\geq k$ simultaneously, and \[p^\lambda\sum_{s\geq k}V^s_{k,1}(\gamma)d_s(\gamma^p)=\tilde{c}_k(\gamma)+d_k(\gamma)\] by the definition of $\tilde{\mathbf{c}}=\mathcal{D}_{\lambda,\tau}(\mathbf{d})$ and that of the map $\mathcal{D}_{\lambda,\tau}$ in Definition~\ref{defn:d-map}. 
For $\gamma\in\mathcal{C}(\tau)$ and $1\leq i \leq p-1$, let $S(\gamma,i):=\bigl\{\alpha\in\tau \ \big| \ \alpha^{p^{h-1}}=\zeta_p^i\gamma\bigr\}.$ Then $\tau_h$ is the disjoint union of the sets $S(\gamma,i)$, and it follows from Lemma~\ref{lem:mahler-coefficients-computations} that, for each $\gamma\in\mathcal{C}(\tau)$ and $1\leq i \leq p-1$, \begin{equation}\label{eq:g1-proof}\sigma^{h-1}\left(\sum_{k\in\mathbb{N}}\frac{\zeta_p^{ik}(\tilde{c}_k(\gamma)+d_k(\gamma))}{(x-\zeta_p^i\gamma)^k}\right)=\sum_{k\in\mathbb{N}}\sum_{\alpha\in S(\gamma,i)}\frac{\sum_{s\geq k}V^s_{k,h-1}(\alpha)\zeta_p^{is}(\tilde{c}_s(\gamma)+d_s(\gamma))}{(x-\alpha)^k}.\end{equation} For each $\alpha\in S(\gamma,i)\Leftrightarrow \alpha^{p^{h-1}}=\zeta_p^i\gamma$, we compute \[\alpha^{p^{h+e-1}}=\left(\alpha^{p^{h-1}}\right)^{p^e}=\left(\zeta_p^i\gamma\right)^{p^e}=\gamma\qquad\text{and}\qquad \zeta_p^{is}=\left(\alpha^{p^{h-1}}\gamma^{-1}\right)^s=\alpha^{sp^{h-1}(1-p^e)},\] and therefore we can rewrite each summand \[V^s_{k,h-1}(\alpha)\zeta_p^{is}(\tilde{c}_s(\gamma)+d_s(\gamma))=V^s_{k,h-1}(\alpha)\alpha^{sp^{h-1}(1-p^e)}\left(\tilde{c}_s\left(\alpha^{p^{h+e-1}}\right)+d_s\left(\alpha^{p^{h+e-1}}\right)\right).\] By Lemma~\ref{lem:mahler-coefficients}, $V^s_{k,h-1}(\alpha)=\mathbb{V}^s_{k,h-1}\cdot\alpha^{k-sp^{h-1}}$, and therefore \[V^s_{k,h-1}(\alpha)\alpha^{sp^{h-1}(1-p^e)}=\mathbb{V}^s_{k,h-1}\cdot\alpha^{k-sp^{h-1}}\cdot\alpha^{sp^{h-1}(1-p^e)}=\mathbb{V}^s_{k,h-1}\alpha^{k-sp^{h+e-1}}.\] Hence \eqref{eq:g1-proof} is equal to \[\sum_{k\in\mathbb{N}}\sum_{\alpha\in S(\gamma,i)}\frac{\sum_{s\geq k}\mathbb{V}^s_{k,h-1}\alpha^{k-sp^{h+e-1}}\left(\tilde{c}_s\left(\alpha^{p^{h+e-1}}\right)+d_s\left(\alpha^{p^{h+e-1}}\right)\right)}{(x-\alpha)^k},\] and our result follows by summing over $\gamma\in\mathcal{C}(\gamma)$ and $1\leq i \leq p-1$.
\end{proof}

\begin{proof}[Proof of Lemma~\protect{\ref{lem:uniform-height}}] First of all, $|\tau_h|=(p^h-p^{h-1})e$ because there are $e$ elements in $\mathcal{C}(\tau)$, each of which has $(p-1)$ distinct $p^\text{th}$ roots (of height $1$) that do not belong to $\mathcal{C}(\tau)$, and each of these latter elements has $p^{h-1}$ distinct $(p^{h-1})^\text{th}$ distinct roots -- it follows from the Definition~\ref{defn:torsion-height} that $\alpha\in\tau$ has height $\eta(\alpha)=h$ if and only if $\alpha$ is a $(p^{h-1})^\text{th}$ root of an element of height $1$. Moreover, the elements $\alpha^{-\lambda}\bar{c}_\lambda(\alpha)$ are all equal to one another if and only if they are all equal to their arithmetic average. So it remains to show that $\alpha^{-\lambda}\bar{c}_\lambda(\alpha)$ is independent of $\alpha$.

Now let $g_\tau\in\mathbb{K}(x)_\tau$ such that $\bar{f}_\tau=\Delta_\lambda(g_\tau)$. By Lemma~\ref{lem:stable-support}(7), $\mathrm{ord}(g,\tau)=\mathrm{ord}(f,\tau)=\lambda$, so we can write \[g_\tau=\sum_{k=1}^\lambda\sum_{n=0}^{h-1}\sum_{\alpha\in\tau_n}\frac{d_k(\alpha)}{(x-\alpha)^k},\] because if $g$ had a pole in $\tau_n$ for some $n\geq h$ then $\Delta_\lambda(g_\tau)=f_\tau$ would have a pole in $\tau_{n+1}$, contradicting our assumptions. Let $\mathbf{d}=(d_k(\gamma))$ for $\gamma$ ranging over $\mathcal{C}(\tau)$ only. Since $\Delta_\lambda(g_\tau)=f_\tau$ has no poles in $\mathcal{C}(\tau)$, we must have $\mathbf{d}\in\mathrm{ker}(\mathcal{D}_{\lambda,\tau})$ by Lemma~\ref{lem:cyclic-component-short}. In particular, for each $\gamma\in\mathcal{C}(\tau)$ we must have \[0=c_\lambda(\gamma)=(\mathcal{D}_{\lambda,\tau}(\mathbf{d}))_{\lambda,\gamma}=-d_\lambda(\gamma)+\sum_{s\geq\lambda}p^\lambda V^s_{\lambda,1}(\gamma)=\gamma^{\lambda-p\lambda}d_\lambda(\gamma^p)-d_\lambda(\gamma),\] since $d_s(\gamma)=0$ for every $s>\lambda$ and $\gamma\in\mathcal{C}(\tau)$ and $V^\lambda_{\lambda,1}(\gamma)=p^{-\lambda}\gamma^{\lambda-p\lambda}$ by Corollary~\ref{cor:mahler-coefficients}, and therefore $\gamma^{-\lambda}d_\lambda(\gamma)=\bar{\omega}$ is a constant that does not depend on $\gamma\in\mathcal{C}(\tau)$. This is the base case $n=0$ of an induction argument showing that $\alpha^{-\lambda}d_\lambda(\alpha)=\bar{\omega}$ is independent of $\alpha\in\tau_n$ for $0\leq n\leq h-1$. Indeed, it follows from Lemma~\ref{lem:mahler-coefficients-computations} and our assumption that $\mathrm{sing}(f,\tau)\cap\mathcal{C}(\tau)=\emptyset$ that
 \begin{multline}\label{eq:uniform-height}\Delta_\lambda\left(\sum_{n=0}^{h-1}\sum_{\alpha\in\tau_n}\frac{d_\lambda(\alpha)}{(x-\alpha)^\lambda}\right)=\sum_{n=0}^{h-1}\sum_{\alpha\in\tau_{n+1}}\frac{p^\lambda V^\lambda_{\lambda,1}(\alpha)d_\lambda(\alpha^p)-d_\lambda(\alpha)}{(x-\alpha)^\lambda}+(\text{lower-order terms})\\ 
=\sum_{n=0}^{h-1}\sum_{\alpha\in\tau_{n+1}}\frac{\alpha^\lambda\cdot ((\alpha^p)^{-\lambda}d_\lambda(\alpha^p))-d_\lambda(\alpha)}{(x-\alpha)^\lambda}+(\text{lower-order terms}) \\ =\sum_{\alpha\in\tau_h}\frac{\bar{c}_\lambda(\alpha)}{(x-\alpha)^\lambda}+(\text{lower-order terms}),
\end{multline} where the second equality follows from the computation $V^{\lambda}_{\lambda,1}(\alpha)=p^{-\lambda}\alpha^{\lambda-p\lambda}$ in Corollary~\ref{cor:mahler-coefficients}. In case $h=1$ we have already concluded our induction argument. In case $h\geq 2$, we proceed with our induction argument and find from \eqref{eq:uniform-height} that we must have \[\alpha^\lambda\cdot ((\alpha^p)^{-\lambda}d_\lambda(\alpha^p))-d_\lambda(\alpha)=0\qquad\Longleftrightarrow\qquad \alpha^{-\lambda}d_\lambda(\alpha)=(\alpha^p)^{-\lambda}d_\lambda(\alpha^p)=\bar{\omega}\] for each $\alpha\in\tau_{n+1}$ whenever $n+1\leq h-1$, since $\alpha^p\in \tau_n$ for such an $\alpha$, concluding our induction argument. Finally, since $d_\lambda(\alpha)=0$ for $\alpha\in\tau_h$, we find again that \[\bar{c}_\lambda(\alpha)=\alpha^\lambda\cdot ((\alpha^p)^{-\lambda}d_\lambda(\alpha^p))=\alpha^\lambda\bar{\omega}\] for $\alpha\in\tau_h$, since $d_\lambda(\alpha)=0$ and $\alpha^p\in\tau_{h-1}$ for such $\alpha$, whence each $d_\lambda(\alpha^p)=\alpha^{p\lambda}\bar{\omega}$. \end{proof}

\subsection{Proof of the Main Theorem}\label{sec:proof} Let us now gather our earlier results into a formal proof of the Main Theorem stated in the introduction, that the $\lambda$-Mahler discrete residue at $\infty$ constructed in Definition~\ref{defn:laurent-residues} for the Laurent polynomial component $f_{\infty}$, together with the $\lambda$-Mahler discrete residues at Mahler trees $\tau\in\mathcal{T}$ constructed in Definition~\ref{defn:non-torsion-residues} for non-torsion $\tau\in\mathcal{T}_0$ and in Definition~\ref{defn:torsion-residues} for torsion $\tau\in\mathcal{T}_+$, comprise a complete obstruction to the $\lambda$-Mahler summability problem.

\begin{mthm} For $\lambda\in\mathbb{Z}$, $f\in\mathbb{K}(x)$ is $\lambda$-Mahler summable if and only if 
\mbox{$\mathrm{dres}_\lambda(f,\infty)=\mathbf{0}$} and $\mathrm{dres}_\lambda(f,\tau,k)=\mathbf{0}$ for every $\tau\in\mathcal{T}$ and every $k\in\mathbb{N}$.
\end{mthm}

\begin{proof} Let $f\in\mathbb{K}(x)$. By Lemma~\ref{lem:rational-decomposition}, $f$ is $\lambda$-Mahler summable if and only if both $f_{\infty}$ and $f_\mathcal{T}$ are Mahler summable. By Proposition~\ref{prop:laurent-residues}, $f_{\infty}$ is $\lambda$-Mahler summable if and only if $\mathrm{dres}(f,\infty)=\mathbf{0}$. By Lemma~\ref{lem:tree-decomposition}, $f_\mathcal{T}$ is $\lambda$-Mahler summable if and only if $f_\tau$ is $\lambda$-Mahler summable for each $\tau\in\mathcal{T}=\mathcal{T}_0\cup\mathcal{T}_+$. By Proposition~\ref{prop:non-torsion-residues} in the non-torsion case $\tau\in\mathcal{T}_0$, and by Proposition~\ref{prop:torsion-residues} in the torsion case $\tau\in\mathcal{T}_+$, $f_\tau$ is $\lambda$-Mahler summable if and only if $\mathrm{dres}_\lambda(f,\tau,k)=\mathbf{0}$ for every $k\in\mathbb{N}$.
\end{proof}


\subsection{Mahler reduction}\label{sec:reduction}

We can now define the $\lambda$-Mahler reduction $\bar{f}_\lambda$ of $f\in\mathbb{K}(x)$ in \eqref{eq:mahler-reduction}, in terms of the local reductions constructed in the proofs of Proposition~\ref{prop:laurent-residues}, Proposition~\ref{prop:non-torsion-residues}, and Proposition~\ref{prop:torsion-residues}: \begin{equation}\label{eq:reduction} \bar{f}_\lambda:=\sum_{\theta\in\mathbb{Z}/\mathcal{P}}\bar{f}_{\lambda,\theta}+\sum_{\tau\in\mathcal{T}}\bar{f}_{\lambda,\tau}=\sum_{\theta\in\mathbb{Z}/\mathcal{P}}\mathrm{dres}_\lambda(f,\infty)_\theta\cdot x^{i_\theta h_\theta(f)}+\sum_{k\in\mathbb{N}}\sum_{\tau\in\mathcal{T}}\sum_{\alpha\in\tau}\frac{\mathrm{dres}_\lambda(f,\tau,k)_\alpha}{(x-\alpha)^k}.\end{equation} We refer to Remark~\ref{rem:laurent-remainder}, Remark~\ref{rem:non-torsion-remainder}, and Remark~\ref{rem:torsion-remainder} for more details.

In the un-twisted case where $\lambda=0$, we had already defined $0$-Mahler discrete residues in \cite{arreche-zhang:2022}, where we proved that they comprise a complete obstruction to what we call here the $0$-Mahler summability problem. That the $\mathrm{dres}(f,\infty)$ of \cite[Def.~4.1]{arreche-zhang:2022} agrees with the $\mathrm{dres}_0(f,\infty)$ of Definition~\ref{defn:laurent-residues} is immediately clear from the formulas. In contrast, the Mahler discrete residues $\mathrm{dres}(f,\tau,k)$ at non-torsion Mahler trees $\tau\in\mathcal{T}_0$ in \cite[Def.~4.10]{arreche-zhang:2022} were defined recursively, using the Mahler coefficients $V^s_{k,1}(\alpha)$ only, whereas here we provide closed formulas using the full set of Mahler coefficients $V^s_{k,n}$ with $n\geq 1$ for $\mathrm{dres}_0(f,\tau,k)$ in Definition~\ref{defn:non-torsion-residues}. Similarly, the Mahler discrete residues at torsion Mahler trees $\tau\in\mathcal{T}_+$ in \cite[Def.~4.16]{arreche-zhang:2022} are defined recursively and in terms of an auxiliary $\mathbb{K}$-linear map (see~\cite[Def.~4.15]{arreche-zhang:2022}), whereas here we provide a closed formulas in terms of a different\footnote{The auxiliary $\mathbb{K}$-linear map in \cite[Def.~4.15]{arreche-zhang:2022} is essentially a truncated version of the map $\mathcal{I}^{(0)}_{0,e}$ of Definition~\ref{defn:omega-section}, in terms of the latter of which we defined $\mathcal{I}^{(0)}_{0,\tau}$ (cf.~Corollary~\ref{cor:inverse-gamma-independent}).} auxiliary $\mathbb{K}$-linear map $\mathcal{I}^{(0)}_{0,\tau}$ in Definition~\ref{defn:torsion-residues}. It is not clear at all (to us) from their respective definitions that the $\mathrm{dres}(f,\tau,k)$ of \cite{arreche-zhang:2022} should agree with the $\mathrm{dres}_0(f,\tau,k)$ defined here. And yet, they do.

\begin{prop}\label{prop:comparison} The Mahler discrete residues $\mathrm{dres}(f,\tau,k)$ of \cite{arreche-zhang:2022} coincide with the $0$-Mahler discrete residues $\mathrm{dres}_0(f,\tau,k)$ in Definitions~\ref{defn:non-torsion-residues} and~\ref{defn:torsion-residues}.
\end{prop}

\begin{proof}
It is clear from \cite[Defs.~4.10 and 4.16]{arreche-zhang:2022} and Definitions~\ref{defn:non-torsion-residues} and \ref{defn:torsion-residues} that the support of both vectors $\mathrm{dres}(f,\tau,k)$ and $\mathrm{dres}_0(f,\tau,k)$ is contained in the set of $\alpha\in\tau$ such that $\eta(\alpha|f)=\mathrm{ht}(f,\tau)$ in the non-torsion case (see Definition~\ref{defn:non-torsion-height}) and such that $\eta(\alpha)=\mathrm{ht}(f,\tau)$ in the torsion case (see Definition~\ref{defn:torsion-height}). In the torsion case $\tau\in\mathcal{T}_+$ such that $\mathrm{ht}(f,\tau)=0$, it is immediately clear from the definitions that $\mathrm{dres}(f,\tau,k)=\mathrm{dres}_0(f,\tau,k)$, so we can assume without loss of generality that either $\tau\in\mathcal{T}_0$ or $\mathrm{ht}(f,\tau)\geq 1$. In \cite[Equation~(4.16)]{arreche-zhang:2022} we constructed a Mahler reduction \[\bar{f}_\tau=\sum_{k\in\mathbb{N}}\sum_{\alpha\in\tau} \frac{\mathrm{dres}(f,\tau,k)_\alpha}{(x-\alpha)^k}\] such that $\bar{f}_\tau-f$ is Mahler summable (see \cite[\S4.4]{arreche-zhang:2022}), whereas here we have constructed an analogous $\bar{f}_{0,\tau}$ in \eqref{eq:reduction} with the same property that $\bar{f}_{0,\tau}-f_\tau$ is $0$-Mahler summable. Therefore \[(\bar{f}_{0,\tau}-f_\tau)-(\bar{f}_\tau-f_\tau)=\bar{f}_{0,\tau}-\bar{f}_\tau=\sum_{k\in\mathbb{N}}\sum_{\alpha\in\tau}\frac{\mathrm{dres}_0(f,\tau,k)_\alpha-\mathrm{dres}(f,\tau,k)_\alpha}{(x-\alpha)}\] is $0$-Mahler summable. If we had $\bar{f}_{0,\tau}\neq \bar{f}_\tau$ then $\mathrm{disp}(\bar{f}_{0,\tau}-\bar{f}_\tau,\tau)=0$ would contradict Theorem~\ref{thm:summable-dispersion}, so we conclude that $\mathrm{dres}_0(f,\tau,k)=\mathrm{dres}(f,\tau,k)$ for every $\tau\in\mathcal{T}$ and $k\in\mathbb{N}$.
\end{proof}


\section{Differential relations among solutions of first-order Mahler equations} \label{sec:telescoping} Let us now consider the differential structures that we shall consider for the most immediate applications of our $\lambda$-Mahler discrete residues. 
We denote by \[\partial:=x\frac{d}{dx}\] the unique $\mathbb{K}$-linear derivation on $\mathbb{K}(x)$ such that $\partial(x)=x$. We immediately compute that $p\sigma\circ\partial=\partial\circ\delta$ as derivations on $\mathbb{K}(x)$. In order to remedy this, one can proceed as proposed by Michael Singer (see \cite{dhr-m}), to work in the overfield $\mathbb{K}(x,\log x)$ and introduce the derivation \[\delta=x\log x \frac{d}{dx}=\log x\cdot \partial.\] We insist that the notation $\log x$ is meant to be suggestive only: here $\log x$ is a new transcendental element satisfying $\sigma(\log x)=p\cdot\log x$ and $\partial(\log x)=1$. Using these properties alone, one can verify that $\delta\circ\sigma=\sigma\circ\delta$ as derivations on all of $\mathbb{K}(x,\log x)$.

The following computational result is a Mahler analogue of \cite[Lem.~3.4]{ArrecheZhang2020}, and of an analogous and more immediate computation in the shift case, which occurs in the proof of \cite[Cor.~2.1]{arreche:2017}. We wish to emphasize that the computation is actually quite straightforward in the case of $\lambda$-Mahler discrete residues at non-torsion Mahler trees $\tau\in\mathcal{T}_0$, and in contrast, rather involved for torsion Mahler trees $\tau\in\mathcal{T}_+$, to to the additional ingredients involved in that case.

\begin{lem}\label{lem:nishioka}
    Let $0\neq a\in\mathbb{K}(x)$. For $\lambda\geq 1$, $\tau\in\mathcal{T}$, and $\alpha\in\tau$, \[\mathrm{dres}_\lambda\left(\partial^{\lambda-1}\left(\frac{\partial(a)}{a}\right),\tau,\lambda\right)_\alpha=(-1)^{\lambda-1}(\lambda-1)!\alpha^{\lambda-1}\cdot\mathrm{dres}_1\left(\frac{\partial(a)}{a},\tau,1\right)_\alpha\in\mathbb{Q}\cdot\alpha^\lambda.\]
\end{lem}

\begin{proof}
    Let $a=b\prod_{\alpha\in\mathbb{K}}(x-\alpha)^{m(\alpha)}$, where $0\neq b\in\mathbb{K}$ and $m(\alpha)\in\mathbb{Z}$, almost all zero, and let \[f:=\frac{\partial(a)}{a}=c(0)+\sum_{\alpha\in\mathbb{K}^\times}\frac{m(\alpha)x}{x-\alpha}=\sum_{\alpha\in\mathbb{K}}m(\alpha)+\sum_{\alpha\in\mathbb{K}^\times}\frac{\alpha\cdot m(\alpha)}{x-\alpha}.\] Then we compute, using a similar induction argument as in \cite[Lem.~3.4]{ArrecheZhang2020}, that for $\tau\in\mathcal{T}$ and $\lambda\geq 1$: \begin{equation}\label{eq:iterated-derivative}\partial^{\lambda-1}(f)_\tau=\sum_{\alpha\in\tau}\frac{(-1)^{\lambda-1}(\lambda-1)!\alpha^\lambda m(\alpha)}{(x-\alpha)^\lambda}+(\text{lower-order terms})=\sum_{k=1}^\lambda\sum_{\alpha\in\tau}\frac{c^{[\lambda]}_k(\alpha)}{(x-\alpha)^k},
    \end{equation} where the notation $c^{[\lambda]}_k(\alpha)$ is meant to let us directly apply the definitions of $\lambda$-Mahler discrete residues of degree $\lambda$ of $\partial^{\lambda-1}(f)$ and more easily compare them with one another. In fact, as we shall see, we will only need to know that $c^{[1]}_1(\alpha)=\alpha\cdot m(\alpha)$, and more generally \begin{equation} \label{eq:c-lambda-1} c^{[\lambda]}_\lambda(\alpha)=(-1)^{\lambda-1}(\lambda-1)!\alpha^{\lambda}m(\alpha)=(-1)^{\lambda-1}(\lambda-1)!\alpha^{\lambda-1}c^{[1]}_1(\alpha).\end{equation} 
    We shall also repeatedly use the results from Lemma~\ref{lem:mahler-coefficients} and Corollary~\ref{cor:mahler-coefficients}, that \[V^\lambda_{\lambda,n}(\alpha)=\mathbb{V}^\lambda_{\lambda,n}\alpha^{\lambda-\lambda p^n}=p^{-\lambda n}\alpha^{\lambda - \lambda p^n},\] without further comment.

    For $\tau\in\mathrm{supp}(f)\cap\mathcal{T}_0$, let $h:=\mathrm{ht}(f,\tau)$, and let $\alpha\in\beta(f,\tau)$ such that $\eta(\alpha|f)=h$ (cf.~Definition~\ref{defn:non-torsion-height}). Then by Definition~\ref{defn:non-torsion-residues}\begin{multline*}\mathrm{dres}_\lambda(\partial^{\lambda-1}(f),\tau,\lambda)_\alpha=\sum_{n=0}^hp^{\lambda n} V^\lambda_{\lambda,n}(\alpha) c^{[\lambda]}_\lambda(\alpha^{p^n})=\sum_{n=0}^hp^{\lambda n}p^{-n\lambda}\alpha^{\lambda-\lambda p^n} c^{[\lambda]}_\lambda(\alpha^{p^n})\\
    =(-1)^{\lambda-1}(\lambda-1)!\alpha^{\lambda}\sum_{n=0}^h m(\alpha^{p^n})=(-1)^{\lambda-1}(\lambda-1)!\alpha^{\lambda-1}\mathrm{dres}_1(f,\tau,1)_\alpha\in\mathbb{Q}\cdot\alpha^\lambda.
    \end{multline*}

    For $\tau\in\mathrm{supp}(f)\cap\mathcal{T}_+$, let us first suppose $\mathrm{ht}(f,\tau)=0$ as in Definition~\ref{defn:torsion-height}, and compute immediately for $\gamma\in\mathcal{C}(\tau),$
    \[\mathrm{dres}_\lambda(\partial^{\lambda-1}(f),\tau,\lambda)_\gamma=c^{[\lambda]}_\lambda(\gamma)=(-1)^{\lambda-1}(\lambda-1)!\gamma^\lambda m(\gamma)=(-1)^{\lambda-1}(\lambda-1)!\gamma^{\lambda-1}\mathrm{dres}_1(f,\tau,1)_\gamma,\] which clearly belongs to $\mathbb{Q}\cdot\gamma^\lambda$.
    On the other hand, if $h:=\mathrm{ht}(f,\tau)\geq 1$, we compute for $\gamma\in\mathcal{C}(\tau)$ using \eqref{eq:torsion-residues-cycle}
    \begin{multline}\label{eq:iterated-derivative-cycle}
    \mathrm{dres}_\lambda(\partial^{\lambda-1}(f),\tau,\lambda)_\gamma=
    \frac{\gamma^\lambda}{e}\sum_{j=1}^e\gamma^{-\lambda p^j} c^{[\lambda]}_\lambda(\gamma^{p^j})=\frac{\gamma^\lambda}{e}\sum_{j=1}^e\gamma^{-\lambda p^j}(-1)^{\lambda-1}(\lambda-1)!\gamma^{\lambda p^j}m(\gamma^{p^j}) \\ =(-1)^{\lambda-1}(\lambda-1)!\frac{\gamma^\lambda}{e}\sum_{j=1}^e m(\gamma^{p^j})=(-1)^{\lambda-1}(\lambda-1)!\gamma^{\lambda-1}\mathrm{dres}_1(f,\tau,1)_\gamma \in\mathbb{Q}\cdot\gamma^\lambda
    \end{multline} Before computing the $\alpha$-component of $\mathrm{dres}_\lambda(\partial^{\lambda-1}(f),\tau,\lambda)$ for $\alpha\in\tau$ such that $\eta(\alpha)=h$, 
    we must first compute a few preliminary objects (cf.~Remark~\ref{rem:torsion-residues-computation}). Consider the vector $\mathbf{d}^{[\lambda]}:=\mathcal{I}^{(0)}_{\lambda,\tau}(\mathbf{c}^{[\lambda]})$ as in Definition~\ref{defn:omega-section}, 
    and let us compute in particular as in \eqref{eq:omega-section-lambda}: \begin{multline}\label{eq:d-lambda-1}
    d^{[\lambda]}_\lambda(\gamma)=\frac{\gamma^\lambda}{e}\sum_{j=0}^{e-1}(j+1-e)\gamma^{-\lambda p^j}c^{[\lambda]}_\lambda(\gamma^{p^j})=
    \frac{\gamma^\lambda}{e}\sum_{j=0}^{e-1}(j+1-e)\gamma^{-\lambda p^j}\cdot (-1)^{\lambda-1}(\lambda-1)!\gamma^{\lambda p^j} m(\gamma^{p^j}) \\ =(-1)^{\lambda-1}(\lambda-1)!\frac{\gamma^\lambda}{e}\sum_{j=0}^{e-1}(j+1-e)m(\gamma^{p^j})=(-1)^{\lambda-1}(\lambda-1)!\gamma^{\lambda-1} d^{[1]}_1(\gamma)
    \end{multline} The $\lambda$-components of $\tilde{\mathbf{c}}^{[\lambda]}:=\mathcal{D}_{\lambda,\tau}(\mathbf{d}^{[\lambda]})$ are simply given by \[\tilde{c}^{[\lambda]}_\lambda(\gamma)=c^{[\lambda]}_\lambda(\gamma)-\mathrm{dres}_\lambda(\partial^{\lambda-1}(f),\tau,\lambda)_\gamma\]
    by Proposition~\ref{prop:omega-section} and \eqref{eq:torsion-residues-cycle}. Therefore, for each $\gamma\in\mathcal{C}(\tau)$, \begin{equation}\label{eq:c-tilde-d}\tilde{c}^{[\lambda]}_\lambda(\gamma)+d^{[\lambda]}_\lambda(\gamma)=\frac{(-1)^{\lambda-1}(\lambda-1)!\gamma^\lambda}{e}\sum_{j=1}^e(j-e)m(\gamma^{p^j}).     
    \end{equation} With this, we next compute the residual average (cf.~Definition~\ref{defn:omega-average}), for which we compute separately the two long sums appearing in \eqref{eq:omega-average}. First, the sum over elements of positive height \begin{multline}\label{eq:omega-average-1}
        \omega_{\lambda,\tau}^{(+)}(\partial^{\lambda-1}(f))=\frac{1}{(p^h-p^{h-1})e}\sum_{\alpha\in\tau_h}\sum_{n=0}^{h-1}p^{\lambda n}\mathbb{V}^\lambda_{\lambda,n}\alpha^{-\lambda p^n} c^{[\lambda]}_\lambda(\alpha^{p^n})\\=\frac{(-1)^{\lambda-1}(\lambda-1)!}{(p^h-p^{h-1})e}\sum_{\alpha\in\tau_h}\sum_{n=0}^{h-1}m(\alpha^{p^n})=(-1)^{\lambda-1}(\lambda-1)!\cdot \omega^{(+)}_{1,\tau}(f).
    \end{multline} Second, the sum over the elements of zero height
    \begin{multline}\label{eq:omega-average-2}
        \omega_{\lambda,\tau}^{(0)}(\partial^{\lambda-1}(f))=\frac{p^{\lambda(e-1)}}{e}\sum_{\gamma\in\mathcal{C}(\tau)}\mathbb{V}^\lambda_{\lambda,h-1}\gamma^{-\lambda}(\tilde{c}^{[\lambda]}(\gamma)+d^{[\lambda]}_\lambda(\gamma))\\=\frac{(-1)^{\lambda-1}(\lambda-1)!}{e^2}\sum_{\gamma\in\mathcal{C}(\tau)}\sum_{j=1}^e(j-e)m(\gamma^{p^j})=(-1)^{\lambda-1}(\lambda-1)!\cdot \omega_{1,\tau}^{(0)}(f).
    \end{multline}
Now putting together \eqref{eq:omega-average-1} and \eqref{eq:omega-average-2} we obtain \begin{gather}\label{eq:omega-average-lambda-1}\omega_{\lambda,\tau}(\partial^{\lambda-1}(f))=\omega_{\lambda,\tau}^{(+)}(\partial^{\lambda-1}(f))-\omega_{\lambda,\tau}^{(0)}(\partial^{\lambda-1}(f))=(-1)^{\lambda-1}(\lambda-1)!\cdot\omega_{1,\tau}(f),\intertext{where}
\omega_{1,\tau}(f)=\omega_{1,\tau}^{(+)}(f)-\omega_{1,\tau}^{(0)}(f)=\frac{1}{(p^h-p^{h-1})e}\sum_{\substack{\alpha\in\tau\\\eta(\alpha> 0}}m(\alpha)-\frac{e-e^2}{2e^2}\sum_{\gamma\in\mathcal{C}(\tau)}m(\gamma)\in\mathbb{Q}.
\end{gather} 

Since the vector $\mathbf{w}^{(\lambda)}$ of Lemma~\ref{lem:d-map-kernel} satisfies $w^{(\lambda)}_\lambda(\gamma)=\gamma^\lambda=\gamma^{\lambda-1} w^{(1)}_1(\gamma),$ we finally compute
\begin{multline}
    \mathrm{dres}_\lambda(\partial^{\lambda-1}(f),\tau,\lambda)_\alpha=\sum_{n=0}^{h-1}p^{n\lambda}V^\lambda_{\lambda,n}(\alpha)c^{[\lambda]}_\lambda(\alpha^{p^n})\\ -p^{\lambda(h-1)}\mathbb{V}^\lambda_{\lambda,h-1}\alpha^{\lambda-\lambda p^{h+e-1}}(\tilde{c}^{[\lambda]}_\lambda(\alpha^{p^{h+e-1}})+d^{[\lambda]}_\lambda(\alpha^{p^{h+e-1}})+\omega_{\lambda,\tau}(\partial^{\lambda-1}(f)) w^{(\lambda)}_\lambda(\alpha^{p^{h+e-1}}))\\=(-1)^{\lambda-1}(\lambda-1)!\alpha^\lambda\left[\sum_{n=0}^{h-1}m(\alpha^{p^n})-\frac{1}{e}\sum_{j=1}^{e}(j-e)m(\alpha^{p^{h+j-1}})+\omega_{1,\tau}(f)\right]\\ =(-1)^{\lambda-1}(\lambda-1)!\alpha^{\lambda-1}\mathrm{dres}_1(f,\tau,1)_\alpha\in\mathbb{Q}\cdot \alpha^\lambda.
\end{multline}
This concludes the proof of the Lemma.\end{proof}

With this preliminary computation now out of the way, we can prove our first application of $\lambda$-Mahler discrete residues in the following result, which is a Mahler analogue of \cite[Cor.~2.1]{arreche:2017} in the shift case and \cite[Prop.~3.5]{ArrecheZhang2020} in the $q$-dilation case.

\begin{prop}\label{prop:nishioka} Let $U$ be a $\sigma\partial$-$\mathbb{K}(x,\log x)$-algebra such that $U^\sigma=\mathbb{K}$. Let $a_1,\dots,a_t\in\mathbb{K}(x)-\{0\},$ and suppose $y_1,\dots,y_t\in U^\times$ satisfy \[\sigma(y_i)=a_iy_i \qquad\text{for}\qquad i=1,\dots,t.\] Then $y_1,\dots,y_t$ are $\partial$-dependent over $\mathbb{K}(x)$ if and only if there exist $k_1,\dots,k_t\in\mathbb{Z}$, not all zero, and $g\in\mathbb{K}(x)$, such that \begin{equation}\label{eq:prop:nishioka}\sum_{i=1}^tk_i\frac{\partial a_i}{a_i}=p\sigma(g)-g.\end{equation}   
\end{prop}

\begin{proof} First, suppose there exist $k_1,\dots,k_t\in\mathbb{Z}$ and $g\in\mathbb{K}(x)$ satisfying \eqref{eq:prop:nishioka}. Consider \[\sigma\left(\sum_{i=1}^t\frac{\delta y_i}{y_i}-g\log x\right)-\left(\sum_{i=1}^t\frac{\delta y_i}{y_i}-g\log x\right)=\log x\left(\sum_{i=1}^t\frac{\partial a_i}{a_i}-(p\sigma(g)-g)\right)=0,\] and therefore \[\sum_{i=1}^t\frac{\delta y_i}{y_i}-g\log x\in U^\sigma=\mathbb{K},\] and therefore $y_1,\dots,y_t$ are $\delta$-dependent over $\mathbb{K}(x,\log x)$, which is equivalent to them being $\partial$-dependent over $\mathbb{K}(x)$, since $\log x$ is $\partial$-algebraic over $\mathbb{K}(x)$.

Now suppose $y_1,\dots,y_t$ are $\partial$-dependent over $\mathbb{K}(x)$. Then there exist linear differential operators $\mathcal{L}_i\in\mathbb{K}[\delta]$, not all zero, such that \[\sum_{i=1}^t\mathcal{L}_i\left(\frac{\delta(a_i)}{a_i}\right)=\sigma(G)-G\] for some $G\in\mathbb{K}(x,\log x)$. Let $\lambda\geq 1$ be as small as possible such that $\mathrm{ord}(\mathcal{L}_i)\leq \lambda-1$ for every $1\leq i\leq t$. Then we must have \[G=\sum_{\ell=1}^\lambda g_\ell \log^\ell x\qquad\text{with} \qquad g_1,\dots,g_\lambda\in\mathbb{K}(x).\] Moreover, writing each $\mathcal{L}_i=\sum_{j=0}^{\lambda-1}k_{i,j}\delta^j$, we must also have \begin{equation}\label{eq:telescoping-relation}\sum_{i=1}^tk_{i,\lambda-1}\partial^{\lambda-1}\left(\frac{\partial a_i}{a_i}\right)=p^\lambda\sigma(g_\lambda)-g_\lambda.\end{equation}

Without loss of generality we can reduce to the situation where, for each $\tau\in\mathcal{T}_0$ and for each $1\leq i \leq t$ such that $\tau\in\mathrm{supp}(\frac{\partial a_i}{a_i})$ we have the same bouquet $\beta(\frac{\partial a_i}{a_i},\tau)$ (cf.~Definition~\ref{defn:non-torsion-height}, and similarly that for each $\tau\in\mathcal{T}_+$ and for each $1\leq i \leq t$ such that $\tau\in\mathrm{supp}(\frac{\partial a_i}{a_i})$ we have $\mathrm{ht}(\frac{\partial a_i}{a_i},\tau)$ the same constant for each $i=1,\dots,t$. Under these conditions, \eqref{eq:telescoping-relation} implies that \[\sum_{i=1}^tk_{i,\lambda-1}\mathrm{dres}_\lambda\left(\partial^{\lambda-1}\left(\frac{\partial a_i}{a_i}\right),\tau,\lambda\right)=\mathbf{0}\] for every $\tau\in\mathcal{T}$. But by Lemma~\ref{lem:nishioka}, this is equivalent to \[\sum_{i=1}^tk_{i,\lambda-1}\mathrm{dres}_1\left(\frac{\partial a_i}{a_i},\tau,1\right)=\mathbf{0},\] and since each $\mathrm{dres}_1(\frac{\partial a_i}{a_i},\tau,1)_\alpha\in\mathbb{Q}\cdot \alpha$ uniformly in $1\leq i\leq t$ and $\alpha\in\tau$ (again by Lemma~\ref{lem:nishioka}), we may further take the $k_{i,\lambda-1}\in\mathbb{Z}$.
\end{proof}


\section{Examples}

In~\cite[Section 5]{arreche-zhang:2022}, the authors provide two small examples for the $\lambda$-Mahler discrete residues 
for $\lambda = 0$. Here, we illustrate $\lambda$-Mahler discrete residues for $\lambda = \pm 1$ in several examples. 
Example~\ref{EG:summable1} gives a $1$-Mahler summable $f$ in the non-torsion case $\tau \subset \mathcal{T}_0$. Example~\ref{EG:nonsummable1} gives a $1$-Mahler non-summable $f$ in the torsion case $\tau\subset \mathcal{T}_{+}$. 
Moreover, Example~\ref{EG:summable-1} gives a $(-1)$-Mahler summable $f$ in the non-torsion case $\tau \subset \mathcal{T}_{0}$. 
Example~\ref{EG:nonsummable-1} gives a $(-1)$-Mahler non-summable $f$ in the torsion case $\tau \subset \mathcal{T}_{+}$.

\begin{eg} \label{EG:summable1}
Let $p = 3, \lambda = 1$, and $\tau=\tau(2)$. Consider the following $f = f_{\tau}$ with $\mathrm{sing}(f,\tau)=\{2, \sqrt[3]{2}, \zeta_3\sqrt[3]{2}, \zeta_3^2\sqrt[3]{2}\}$ :
\begin{align*}
f & = \frac{-x^6+4 x^3+3 x^2-12 x+8}{(x-2)^2 \left(x^3-2\right)^2} \\
  & = \frac{-1}{(x-2)^2}+\frac{1}{6\sqrt[3]{2}}\cdot\sum_{i=0}^2 \frac{\zeta_3^{2i}}{(x-\zeta_3^i\sqrt[3]{2})^2} -\frac{1}{3\sqrt[3]{4}}\cdot\sum_{i=0}^2 \frac{\zeta_3^i}{x-\zeta_3^i\sqrt[3]{2}} \\
  & = \sum_{k = 1}^2 \sum_{\alpha \in \beta(f, \tau)} \frac{c_k(\alpha)}{(x-\alpha)^k},
\end{align*}
where $\beta(f, \tau) = \{2, \gamma, \zeta_3 \gamma, \zeta_3^2 \gamma \}$ with $\gamma := \sqrt[3]{2}$. By Definition~\ref{defn:non-torsion-height}, we have $\mathrm{ht}(f,\tau) = 1$. 
It follows from Definition~\ref{defn:non-torsion-residues} that for $i \in \{0, 1, 2\}$: 
\begin{align*}
\mathrm{dres}_1(f,\tau,1)_{\zeta_3^i\gamma } 
& = V_{1, 0}^1 (\zeta_3^i \gamma) c_1(\zeta_3^i \gamma) + 3 V_{1, 1}^1(\zeta_3^i \gamma) c_1(2) + V_{1, 0}^2(\zeta_3^i \gamma) c_1(\zeta_3^i \gamma) + 3 V_{1, 1}^2(\zeta_3^i \gamma)c_2(2) \\
& = 1 \cdot (- \frac{\zeta_3^i}{3 \sqrt[3]{4}}) + (-3) \cdot \left(-\frac{\zeta_3^i \gamma}{2 \cdot 3^2}\right) \\
& = 0,
\end{align*}
and 
\begin{align*}
\mathrm{dres}_1(f,\tau,2)_{\zeta_3^i\gamma } 
& = V_{2, 0}^2 (\zeta_3^i \gamma) c_2(\zeta_3^i \gamma) + 3 V_{2, 1}^2(\zeta_3^i \gamma)c_2(2) \\
& = \frac{\zeta_3^{2 i}}{6 \sqrt[3]{2}}+ (-3) \cdot \left(\frac{\zeta_3^{2 i}}{2 \cdot 3^2 \cdot \sqrt[3]{2}}\right) \\
& = 0.
\end{align*}
Thus, we see from Proposition~\ref{prop:non-torsion-residues} that $f$ is $1$-Mahler summable. And indeed, 
$$f  = \Delta_1\left(\frac{1}{(x-2)^2}\right).$$
\end{eg}

 \begin{eg} \label{EG:nonsummable1}
Let $p = 3, \lambda = 1$, and $\tau=\tau(\zeta_4)$. Consider the following $f=f_\tau$ with $\mathrm{sing}(f,\tau)=\{\zeta_4^{\pm 1},\zeta_{12}^{\pm 1},\zeta_{12}^{\pm 5}\}$: 
\begin{align*}
f & = \frac{-2 x^4+2 x^2+1}{\left(x^2+1\right) \left(x^4-x^2+1\right)} \\
  & = \frac{1}{2} \left( -\frac{\zeta_4^3}{x - \zeta_4} - \frac{\zeta_4}{x - \zeta_4^3} + \frac{\zeta_{12}^7}{x - \zeta_{12}} + \frac{\zeta_{12}^{11}}{x - \zeta_{12}^5} + \frac{\zeta_{12}}{x - \zeta_{12}^7} + \frac{\zeta_{12}^5}{x - \zeta_{12}^{11}} \right) \\
  & =  \sum_{\alpha \in \mathrm{sing}(f,\tau)} \frac{c_k(\alpha)}{x-\alpha}.
\end{align*}
By Definition~\ref{defn:torsion-height}, we see that $\mathrm{ht}(f,\tau) = 1$. Furthermore, by Definition~\ref{defn:omega-section}, \ref{defn:omega-average}, and~\ref{defn:d-map}, we 
find that
\begin{align*}
 \omega & := \omega_{1, \tau}(f) = -1/4, \\
 \mathcal{I}_{1, \tau}^{(\omega)}(\mathbf{c})&  = \left(d_1(\zeta_4), d_1(\zeta_4^3)\right) = -\frac{1}{4}(\zeta_4 + \zeta_4^3) \left(1,  1\right), \\
 \mathcal{D}_{1, \tau}(\mathbf{d}) & = \left(\tilde{c}_1(\zeta_4), \tilde{c}_1(\zeta_4^3)\right) = (0, 0). 
\end{align*}
Thus, it follows from Definition~\ref{defn:torsion-residues} that 
\begin{align*}
\mathrm{dres}_1(f, \tau, 1)_{\zeta_{12}} & = V_{1, 0}^1(\zeta_{12}) \cdot c_1(\zeta_{12}) - \mathbb{V}_{1, 0}^1 \cdot (\zeta_{12})^{-8} \cdot d_1(\zeta_{12}^9)  \\
& = c_1(\zeta_{12}) - \zeta_3 \cdot d_1(\zeta_4^3) \\
& = \zeta_{12}^7 - \zeta_3 \cdot (-\frac{1}{4}) \cdot (\zeta_4^3 - \zeta_4) \\
& = \frac{1}{4} \zeta_{12} + \frac{3}{4} \zeta_{12}^7  \neq 0.
\end{align*}
Similarly, a direct calculation shows that 
\begin{align*}
\mathrm{dres}_1(f, \tau, 1)_{\zeta_{12}^7} & =  \frac{3}{4} \zeta_{12} + \frac{1}{4} \zeta_{12}^7  \neq 0, \\
\mathrm{dres}_1(f, \tau, 1)_{\zeta_{12}^5} & =  \frac{1}{4} \zeta_{12}^5 + \frac{3}{4} \zeta_{12}^{11}  \neq 0, \\
\mathrm{dres}_1(f, \tau, 1)_{\zeta_{12}^5} & =  \frac{3}{4} \zeta_{12}^5 + \frac{1}{4} \zeta_{12}^{11}  \neq 0,
\end{align*}
and
\begin{align*}
\mathrm{dres}_1(f, \tau, 1)_{\zeta_{4}} & = c_1(\zeta_4) - \tilde{c}_1(\zeta_4) = c_1(\zeta_4) = -\frac{1}{2} \zeta_4^3 \neq 0, \\
\mathrm{dres}_1(f, \tau, 1)_{\zeta_{4}^3} & =  c_1(\zeta_4^3) - \tilde{c}_1(\zeta_4^3) = c_1(\zeta_4^3) = -\frac{1}{2} \zeta_4 \neq 0.
\end{align*}
Thus, it follows from Proposition~\ref{prop:torsion-residues} that $f$ is not $1$-Mahler summable.
 \end{eg}
 
 \begin{eg} \label{EG:summable-1}
 Let $p = 3, \lambda = -1$, and $\tau=\tau(5)$. Consider the following $f = f_{\tau}$ with $\mathrm{sing}(f,\tau)=\{5, \sqrt[3]{5}, \zeta_3\sqrt[3]{5}, \zeta_3^2\sqrt[3]{5}\}$ :
\begin{align*}
f & =\frac{-3 x^6+30 x^3+x^2-10 x-50}{3 (x-5)^2 \left(x^3-5\right)^2} \\
  & = \frac{-1}{(x-5)^2}+\frac{1}{135 \sqrt[3]{5}}\cdot\sum_{i=0}^2 \frac{\zeta_3^{2i}}{(x-\zeta_3^i\sqrt[3]{5})^2} -\frac{2}{135 \sqrt[3]{25}}\cdot\sum_{i=0}^2 \frac{\zeta_3^i}{x-\zeta_3^i\sqrt[3]{5}} \\
  & = \sum_{k = 1}^2 \sum_{\alpha \in \beta(f, \tau)} \frac{c_k(\alpha)}{(x-\alpha)^k},
\end{align*}
where $\beta(f, \tau) = \{2, \gamma, \zeta_3 \gamma, \zeta_3^2 \gamma \}$ with $\gamma := \sqrt[3]{5}$. By Definition~\ref{defn:non-torsion-height}, we have $\mathrm{ht}(f,\tau) = 1$. 
It follows from Definition~\ref{defn:non-torsion-residues} that for $i \in \{0, 1, 2\}$: 
\begin{align*}
\mathrm{dres}_1(f,\tau,1)_{\zeta_3^i\gamma } 
& = V_{1, 0}^1 (\zeta_3^i \gamma) c_1(\zeta_3^i \gamma) + 3^{-1} V_{1, 1}^1(\zeta_3^i \gamma) c_1(2) + V_{1, 0}^2(\zeta_3^i \gamma) c_1(\zeta_3^i \gamma) + 3^{-1} V_{1, 1}^2(\zeta_3^i \gamma)c_2(2) \\
& = 1 \cdot (- \frac{2 \zeta_3^i}{135 \sqrt[3]{25}}) + (-\frac{1}{3}) \cdot \left(-\frac{ 2\zeta_3^i \gamma}{3^2 \cdot 5^2}\right) \\
& = 0,
\end{align*}
and 
\begin{align*}
\mathrm{dres}_1(f,\tau,2)_{\zeta_3^i\gamma } 
& = V_{2, 0}^2 (\zeta_3^i \gamma) c_2(\zeta_3^i \gamma) + 3^{-1} V_{2, 1}^2(\zeta_3^i \gamma)c_2(2) \\
& = \frac{\zeta_3^{2 i}}{135 \sqrt[3]{2}}+ (-\frac{1}{3}) \cdot \left(\frac{\zeta_3^{2 i}}{3^2 \cdot 5 \cdot \sqrt[3]{5}}\right) \\
& = 0.
\end{align*}
Thus, we see from Proposition~\ref{prop:non-torsion-residues} that $f$ is $(-1)$-Mahler summable. And indeed, 
$$f  = \Delta_{-1}\left(\frac{1}{(x-5)^2}\right).$$
\end{eg}

 \begin{eg} \label{EG:nonsummable-1}
 Let $p = 2, \lambda = -1$, and $\tau=\tau(\zeta_3)$. Consider the following $f=f_\tau$ with $\mathrm{sing}(f,\tau)=\{\zeta_3^{\pm 1},\zeta_{6}^{\pm 1}\}$: 
\begin{align*}
f & = \frac{1}{2 \left(x^4+x^2+1\right)} \\
  & = -\frac{1}{2} \left( \frac{\zeta_3}{x - \zeta_3} + \frac{\zeta_3^{-1}}{x - \zeta_3^{-1}} + \frac{\zeta_{6}}{x - \zeta_{6}} + \frac{\zeta_{6}^{-1}}{x - \zeta_{6}^{-1}} \right) \\
  & = \sum_{\alpha \in \mathrm{sing}(f,\tau)} \frac{c_k(\alpha)}{x-\alpha}.
\end{align*}
By Definition~\ref{defn:torsion-height}, we see that $\mathrm{ht}(f,\tau) = 1$. Furthermore, by Definition~\ref{defn:omega-section}, \ref{defn:omega-average}, and~\ref{defn:d-map}, we 
find that
\begin{align*}
 \omega & := \omega_{1, \tau}(f) = 0, \\
 \mathcal{I}_{1, \tau}^{(\omega)}(\mathbf{c})&  = \left(d_1(\zeta_3), d_1(\zeta_3^{-1})\right) = \frac{2}{3} \left(\zeta_3,  \zeta_3^{-1}\right), \\
 \mathcal{D}_{1, \tau}(\mathbf{d}) & = \left(\tilde{c}_1(\zeta_3), \tilde{c}_1(\zeta_3^{-1})\right) = -\frac{1}{3} \left(\zeta_3,  \zeta_3^{-1}\right). 
\end{align*}
Thus, it follows from Definition~\ref{defn:torsion-residues} that 
\begin{align*}
\mathrm{dres}_1(f, \tau, 1)_{\zeta_{6}} & = V_{1, 0}^1(\zeta_{6}) \cdot c_1(\zeta_{6}) - \mathbb{V}_{1, 0}^1 \cdot (\zeta_{6})^{-3} \cdot \left(\tilde{c}_1(\zeta_3^{-1}) + d_1(\zeta_{3}^{-1}) \right) \\
& = c_1(\zeta_6) + \tilde{c}_1(\zeta_3^{-1}) + d(\zeta_3^{-1})\\
& = -\frac{1}{2}\zeta_6 - \frac{1}{3} \zeta_3^{-1} + \frac{2}{3} \zeta_3^{-1} \\
& = \frac{1}{3} \zeta_3^{-1} - \frac{1}{2} \zeta_6 \neq 0.
\end{align*}
Similarly, a direct computation shows that 
\[
\mathrm{dres}_1(f, \tau, 1)_{\zeta_{6}^{-1}} = \frac{1}{3} \zeta_3 - \frac{1}{2} \zeta_6^{-1} \neq 0. 
\]

Therefore, it follows from Proposition~\ref{prop:torsion-residues} that $f$ is not $(-1)$-Mahler summable.
 \end{eg}
 
 \bibliographystyle{alpha}

\end{document}